\documentclass[a4paper,10pt,reqno,draft]{amsart}

\usepackage{amsthm, amssymb, amsmath, latexsym}
\usepackage{amsfonts, mathrsfs, color}
\usepackage{a4wide}
\usepackage{listings}
\usepackage{color}
\usepackage[usenames,dvipsnames]{xcolor}
\usepackage{rotating}
\usepackage{hhline}
\usepackage{subfigure}
\usepackage{multirow}
\usepackage{enumerate}
\usepackage[bookmarks]{}
\usepackage{amsfonts}   % if you want the fonts
\usepackage{dsfont}
\usepackage{tikz}
\usepackage{accents}
 \theoremstyle{plain}
 \begingroup
 \newtheorem{thm}{Theorem}[section]
 \newtheorem{lem}[thm]{Lemma}
 \newtheorem{prop}[thm]{Proposition}
 \newtheorem{cor}[thm]{Corollary}
 \endgroup

 \theoremstyle{definition}
 \begingroup
 \newtheorem{defn}[thm]{Definition}
 
 \endgroup

 \theoremstyle{remark}
 \begingroup
 \newtheorem{rem}[thm]{Remark}
 \endgroup

 \numberwithin{equation}{section}

\mathchardef\emptyset="001F

\newcommand{\hs}{{\mathcal H}}

\newcommand{\ds}{\displaystyle}

\newcommand{\G}{{\mathcal G}}
\newcommand{\dx}{\,dx}

\newcommand{\I}{\mathscr{I}}
\newcommand{\B}{\mathscr{B}}
\newcommand{\T}{\mathcal{T}}
\newcommand{\mI}{\mathcal{I}}

\newcommand{\R}{{\mathbb R}}
\newcommand{\Z}{{\mathbb Z}}
\newcommand{\Q}{{\mathbb Q}}
\newcommand{\N}{{\mathbb N}}
\newcommand{\Sph}{{\mathbb S}}

\newcommand{\Om}{\Omega}

\newcommand{\e}{\varepsilon}

\newcommand{\om}{\omega}
\newcommand{\ie}{{; \it i.e., }}

\newcommand{\A}{\mathscr{A}}

\title[]
{Stochastic Homogenisation of Free-Discontinuity Problems}

\author[]{FILIPPO CAGNETTI}
\address[]{Department of Mathematics,
University of Sussex, Brighton, United Kingdom}
\email[]{F.Cagnetti@sussex.ac.uk}
\author[]{GIANNI DAL MASO}
\address[]{SISSA, Trieste, Italy}
\email[]{dalmaso@sissa.it}
\author[]{LUCIA SCARDIA}
\address[]{Department of Mathematical Sciences, University of Bath, United Kingdom}
\email[]{L.Scardia@bath.ac.uk}
\author[]{CATERINA IDA ZEPPIERI}
\address[]{Angewandte Mathematik,
WWU M\"unster, Germany}
\email[]{caterina.zeppieri@uni-muenster.de}
\begin{document}

%%----------------------------------------------------------------------------------------------------------------------

\begin{abstract}
In this paper we study the stochastic homogenisation of \emph{free-discontinuity functionals}.  
Assuming stationarity for the random volume and surface integrands, we prove the existence of a homogenised random free-discontinuity functional, which is deterministic in the ergodic case. Moreover, by establishing a connection between the deterministic convergence of the functionals at any fixed realisation and the pointwise Subadditive Ergodic Theorem by Akcoglou and Krengel, we characterise the limit volume and surface integrands in terms of asymptotic cell formulas.% involving minimisation problems on increasingly larger cubes. 
\end{abstract}

\maketitle

{\small
\keywords{\textbf{Keywords:} Subadditive Ergodic Theorem, stochastic homogenisation, free-discontinuity problems, $\Gamma$-convergence.

\medskip


\bigskip

\section{Introduction}

In this article we prove a stochastic homogenisation result for sequences of \emph{free-discontinuity functionals} of the form 
\begin{equation}\label{Intro:funct-e}
E_\e(\omega)(u)= \int_A f\left(\omega,\frac{x}{\e},\nabla u\right)\dx+\int_{S_u\cap A}g\left(\omega,\frac{x}{\e},u^+-u^-,\nu_u\right)\,d\mathcal H^{n-1},
\end{equation}
where $f$ and $g$ are random integrands, $\om$ is the random parameter, %describing the environment
and $\e>0$ is a small scale parameter. The functionals $E_\e$ are defined in the space $SBV(A, \R^m)$ of special
 $\R^m$-valued  functions of bounded variation on the open set  $A \subset \R^n$. This space was introduced
 by De Giorgi and Ambrosio in \cite{DGA} to deal with \emph{deterministic} problems -  \emph{e.g.} in fracture mechanics, image segmentation, or in the study of liquid crystals - where the variable $u$ can have discontinuities on a hypersurface which is not known a priori, hence the name free-discontinuity functionals \cite{DG}. In \eqref{Intro:funct-e}, $S_u$ denotes the %free 
discontinuity set of $u$, $u^+$ and $u^-$ are the ``traces'' of $u$ on both sides of $S_u$, $\nu_u$ denotes the (generalised) normal to $S_u$, and $\nabla u$ denotes the approximate differential of $u$.
\smallskip

Our main result is that, in the macroscopic limit $\e\to 0$, the functionals $E_\e$ homogenise to a stochastic \emph{free-discontinuity functional} of the same form, under the assumption that $f$ and $g$ are stationary with respect to $\om$, and that each of the realisations  $f(\om,\cdot,\cdot)$  and  $g(\om, \cdot,\cdot,\cdot)$ satisfies the hypotheses considered  in the deterministic  case studied in  \cite{CDMSZ} (see Section \ref{S:setting} for details). Moreover, we show that under the additional assumption of ergodicity of $f$ and $g$ the homogenised
limit of $E_\e$ is deterministic. Therefore, our qualitative homogenisation result extends to the $SBV$-setting the classical qualitative results by Papanicolaou and Varadhan \cite{PV1,PV2}, Kozlov \cite{Koz1}, and Dal Maso and Modica \cite{DMM1,DMM2}, which were formulated in the more regular Sobolev setting.

\subsection{A brief literature review.} The study of variational limits of \emph{random free-discontinuity} functionals is very much at its infancy. To date, the only available results are limited to the special case of discrete energies of spin systems  \cite{Alciru, BP}, where the authors consider purely \textit{surface} integrals, and $u$ is defined on a discrete lattice and takes values in $\{\pm1\}$.

In the case of \emph{volume} functionals in Sobolev spaces, classical qualitative results are provided by the work by Papanicolaou and Varadhan \cite{PV1,PV2} and Kozlov \cite{Koz1} in the linear case, and by Dal Maso and Modica \cite{DMM1,DMM2} in the nonlinear setting. 
The need to develop efficient methods to determine the homogenised coefficients and to estimate the error in the homogenisation approximation, has recently motivated an intense effort to build a quantitative theory of stochastic homogenisation in the regular Sobolev case.

The first results in this direction are due to Gloria and Otto in the discrete setting \cite{GO1, GO2}. In the continuous setting, quantitative estimates
for the convergence results
% of Dal Maso and Modica
are given by Armstrong and Smart \cite{AS}, who also study the regularity of the minimisers, and by Armstrong, Kuusi, and Mourrat \cite{AKM1, AKM2}. We also mention \cite{AM}, where Armstrong and Mourrat give Lipschitz regularity for the solutions of elliptic equations with random coefficients, by directly studying certain functionals
that are minimised by the solutions. 

\smallskip

The mathematical theory of \emph{deterministic homogenisation} of free-discontinuity problems is well established. When $f$ and $g$ are \textit{periodic} in the spatial variable, the limit behaviour of $E_\e$ can be determined by classical homogenisation theory. In this case, under mild assumptions on $f$ and $g$, the deterministic functionals $E_\e$ behave macroscopically like a homogeneous free-discontinuity functional. If, in addition, the integrands $f$ and $g$ satisfy some standard growth and coercivity conditions, the limit behaviour of $E_\e$ is given by the simple superposition of the limit  behaviours  of its volume and surface parts (see \cite{BDfV}). This is, however, not always the case if $f$ and $g$ satisfy ``degenerate'' coercivity conditions. Indeed, while in \cite{BF, CS, FGP} the two terms in $E_\e$ do not interact, in \cite{BDM, DMZ, BLZ, LS1, LS2} they do interact and produce rather complex limit effects. The study of the deterministic homogenisation of free-discontinuity functionals without any periodicity  condition, and under general assumptions ensuring that the volume and surface terms do ``not mix'' in the limit, has been recently carried out in \cite{CDMSZ}.

\smallskip

\subsection{Stationary random integrands.}%\label{Sect:model} 
Before giving the precise statement of our results, we need to recall some definitions.
The random environment is modelled by a probability space $(\Omega, \T , P)$ endowed with a group $\tau=(\tau_z)_{z\in \Z^n}$  of $\T$-measurable $P$-preserving transformations  on $\Om$. That is, the action of $\tau$ on $\Omega$ satisfies 
$$
P(\tau ( E)) = P(E) \quad \text{ for every } E\in \T.
$$
We say that $f\colon \Omega \times \R^n \times \R^{m\times n} \to [0,+\infty)$ and $g\colon \Omega \times \R^n \times (\R^m\setminus \{0\} ) \times \mathbb{S}^{n-1} \to [0,+\infty)$ are \emph{stationary random volume and surface integrands} if they  satisfy the assumptions introduced in the deterministic work \cite{CDMSZ} (see  Section \ref{S:setting} for the complete list of assumptions) for every realisation, and the following stationarity condition with respect to $\tau$:
\begin{align*}%\label{Intro:f}
f(\omega, x+z,\xi)&=f(\tau_z( \omega), x,\xi) \qquad \text{ for every } (x,\xi)\in \R^n \times \R^{m\times n}\; \text{and every}\; z\in  \Z^n,
\\
%\label{Intro:g}
g(\omega, x+z,\zeta,\nu)&=g(\tau_z( \omega), x,\zeta,\nu) \quad \text{ for every } (x,\zeta,\nu)\in \R^n \times ( \R^m\setminus\{0\} ) \times \mathbb{S}^{n-1} \text{and every}\; z\in  \Z^n.&
\end{align*}
When, in addition, $\tau$ is \emph{ergodic}, namely when any $\tau$-invariant set $E\in  \T $ has probability zero or one, we say that $f$ and $g$ are ergodic. 

\smallskip

\subsection{The main result: Method of proof and comparison with previous works.} Under the assumption that $f$ and $g$ are stationary random integrands, we prove the convergence of $E_\e$ to a random homogenised functional $E_{\rm hom}$ (Theorem  \ref{G-convE}), and we provide representation formulas for the limit volume and surface integrands (Theorem \ref{en-density_vs}). The combination of these two results shows, in particular, that the limit functional $E_{\rm hom}$ is a free-discontinuity functional of the same form as $E_\e$. If, in addition, $f$ and $g$ are ergodic, we show that $E_{\rm hom}$ is  deterministic.

\smallskip

Our method of proof consists of two main steps: a purely deterministic step and a stochastic one, in the spirit of the strategy introduced in \cite{DMM2} for integral functionals of volume type defined on Sobolev spaces.

\smallskip

In the \emph{deterministic step} we fix $\omega \in \Om$ and we study the asymptotic behaviour of $E_\e(\omega)$. 
Our recent result \cite[Theorem 3.8]{CDMSZ} ensures that $E_\e(\omega)$ converges (in the sense of $\Gamma$-convergence) to a free-discontinuity functional of the form
$$
E_{\rm hom}(\om)(u)=\int_A f_{\rm hom}\left(\omega,\nabla u\right)\dx+\int_{S_u\cap A}g_{\rm hom}\left(\omega,[u],\nu_u\right)\,d\mathcal H^{n-1},
$$
with 
\begin{align}\label{i:f-hom}
f_{\rm hom}(\omega, \xi)&:= \lim_{r\to 0^+}\frac{1}{r^n}\inf \int_{Q_r(rx)} f(\omega, y,\nabla u(y))dy, \\
\label{i:g-hom}
g_{\rm hom}(\omega, \zeta, \nu)&:=\lim_{r\to 0^+}\frac{1}{r^{n-1}}\inf \int_{S_u\cap Q^\nu_r(rx)} g(\omega,y,[u](y),\nu_u(y))d \mathcal{H}^{n-1}(y),
\end{align}
provided the limits in \eqref{i:f-hom}-\eqref{i:g-hom} exist and are independent of $x$.
In \eqref{i:f-hom}  the infimum is taken among Sobolev functions attaining the linear boundary datum $\xi x$ near $\partial Q_r(rx)$ (see \eqref{candidate:v} below), where $Q_r(rx) = rQ(x)$ is the blow-up by $r$ of the unit cube centred at $x$. In \eqref{i:g-hom}  the infimum is taken among all Caccioppoli partitions (namely $u \in SBV_{\textrm{pc}}(Q^\nu_r(rx),\R^m)$, see (f) in Section~\ref{Notation}) attaining a piecewise constant boundary datum near $\partial Q^\nu_r(rx)$ (see \eqref{I:min-pb-sur}), and $Q^\nu_r(rx)$ is obtained by rotating $Q_r(rx)$ in such a way that one face is perpendicular to $\nu$.

In the \emph{stochastic step} we prove that the limits \eqref{i:f-hom} and \eqref{i:g-hom} exist \emph{almost surely} and are independent of $x$. 
To this end, it is crucial to show that we can apply the pointwise Subadditive Ergodic Theorem by Akcoglou and Krengel \cite{AK}. Since our  convergence result \cite{CDMSZ} ensures that there is no interaction between  the volume and surface terms in the limit, we can treat them separately. 

More precisely, for the volume term, proceeding as in  \cite{DMM2} (see also \cite{MeMi}), one can show that the map
\begin{equation}\label{candidate:v}
(\om,Q)\mapsto \inf \bigg\{\int_{Q} f(\omega, y,\nabla u(y))dy \colon u\in W^{1,p}(Q,\R^m), \; u (y)= \xi y \textrm{ near } \partial Q\bigg\}
\end{equation}
defines a \emph{subadditive stochastic process} for every fixed $\xi\in \R^{m\times n}$ (see Definition \ref{Def:subadditive}).
Then the almost sure existence of the limit of \eqref{i:f-hom} and its independence of $x$ directly follow by the $n$-dimensional pointwise Subadditive Ergodic Theorem, which also ensures that the limit is deterministic if $f$ is ergodic.  

For the surface term, however, even applying this general programme presents several difficulties. One of the obstacles is due to a nontrivial ``mismatch'' of dimensions:   On  the one hand the minimisation problem 
\begin{equation}\label{I:min-pb-sur}
\inf \bigg\{ \int_{S_u\cap Q^\nu_r(rx)} \!\!\!\!\! \!\!\!\!\! \!\!\!\!\!  g(\omega,y,[u],\nu_u)d \mathcal{H}^{n-1} \colon u \in SBV_{\textrm{pc}}(Q^\nu_r(rx),\R^m),\, u=u_{rx,\zeta,\nu}\; \text{on}\; \partial Q^\nu_r(rx)\bigg\}
\end{equation}
appearing in \eqref{i:g-hom} is defined on the $n$-dimensional set $Q^\nu_r(rx)$; on the other hand the integration is performed on the $(n-1)$-dimensional set 
$S_u\cap Q^\nu_r(rx)$ and the integral rescales   in $r$  like a surface measure. In other words, the surface term is an $(n-1)$-dimensional measure which is naturally defined on $n$-dimensional sets.
Understanding how to match these different dimensions is a key preliminary step to define a suitable \emph{subadditive stochastic process} for the application of the Subadditive Ergodic Theorem in dimension $n-1$.  

To this end we first set $x=0$. We want to consider the infimum in \eqref{I:min-pb-sur} as a function of $\left(\om, I\right)$, where $I$ belongs to the class $\mathcal{I}_{n-1}$ of $(n-1)$-dimensional intervals (see \eqref{int:int}). To do so, we  define a systematic way to ``complete'' the missing dimension and to rotate the resulting $n$-dimensional interval.  For this we proceed as in \cite{Alciru}, where the authors had to face a similar problem in the study of pure surface energies of spin systems. 

Once this preliminary problem is overcome, we prove in Proposition \ref{propr_ms} that the infimum in \eqref{I:min-pb-sur} with $x=0$ and $\nu$ with rational coordinates is related to an $(n-1)$-dimensional  subadditive stochastic process $\mu_{\zeta,\nu}$ on $\Omega\times \mathcal{I}_{n-1}$  with respect to a suitable group $(\tau_{z'}^{\nu})_{{z'}\in \Z^{n-1}}$ of $P$-preserving transformations  (see Proposition \ref{propr_ms}). A key difficulty in the proof is to establish the measurability  in $\omega$ of the infimum \eqref{I:min-pb-sur}. Note that this is clearly not an issue in the case of volume integrals considered in \cite{DMM1, DMM2}: The infimum in \eqref{candidate:v} is computed on a separable space, so it can be done over a countable set of functions, and hence the measurability of the process follows directly from the measurability of $f$. This is not an issue for the surface energies considered in \cite{Alciru} either: Since the problem is studied in a discrete lattice, the minimisation is reduced to a countable  collection  of functions. The infimum in \eqref{I:min-pb-sur}, instead, cannot be reduced to a countable set, hence the proof of measurability is not straightforward (see Proposition \ref{measurability} in the Appendix).

\smallskip

The next step is to apply the $(n-1)$-dimensional  Subadditive Ergodic Theorem to the subadditive stochastic process $\mu_{\zeta,\nu}$, for fixed $\zeta$ and $\nu$. This ensures that the limit 
\begin{equation}\label{lim:density-intro}
g_{\zeta,\nu}(\om):=\lim_{t\to +\infty} \frac{\mu_{\zeta,\nu}(\omega)(tI)}{t^{n-1}\mathcal{L}^{n-1}(I)},
\end{equation} 
exists for $P$-a.e.\ $\om\in\Om$ and does not depend on $I$. The fact that the limit in \eqref{lim:density-intro} exists in a set of full measure, common to every $\zeta$ and $\nu$, requires some attention (see Proposition \ref{Ergodic-thm-surf}), 
and follows from the continuity properties in $\zeta$ and $\nu$ of some auxiliary functions 
(see \eqref{C:g-utilde} and \eqref{C:g-tilde} in Lemma~\ref{continuity_in_zeta-nu}).

As a final step, we need to show that the limit in \eqref{i:g-hom} is independent of $x$, namely that the choice $x=0$ is not restrictive. We remark that the analogous result for \eqref{i:f-hom} follows directly by $\Gamma$-convergence and by the Subadditive Ergodic Theorem (see also \cite{DMM2}). The surface case, however, is more subtle, since the minimisation problem in \eqref{I:min-pb-sur} depends on $x$ also through the boundary datum $u_{rx,\zeta,\nu}$. To prove the $x$-independence of $g_{\rm hom}$ we proceed in three steps. First, we exploit the stationarity of $g$ to show that \eqref{lim:density-intro} is $\tau$-invariant. Then, we prove the result when $x$ is integer, by combinining the Subadditive Ergodic Theorem and the Birkhoff Ergodic Theorem, in the spirit of \cite[Proof of Theorem 5.5]{Alciru} (see also \cite[Proposition 2.10]{BP}). Finally, we conclude the proof with a careful approximation argument.

\subsection{Outline of the paper.} The paper is organised as follows. In Section \ref{Notation} we introduce some notation used throughout the paper. In the first part of Section \ref{S:setting} we state the assumptions on $f$ and $g$ and we introduce the stochastic setting of the problem; the second part is devoted to the statement of the main results of the paper. The behaviour of the volume term is studied in the short Section \ref{Ergodic}, while Sections \ref{existence ghom} and \ref{not special}, as well as the Appendix, deal with the surface term.

\section{Notation}\label{Notation}

We introduce now some notation that will be used throughout the paper. For the convenience of the reader 
we follow the ordering used in \cite{CDMSZ}. 
\begin{itemize}
\item[(a)] $m$ and $n$ are fixed positive integers, with $n\ge 2$, $\R$ is the set of real numbers, and  $\R^m_0:=\R^m\setminus\{0\}$, 
while $\Q$ is the set of rational numbers and $\Q^m_0:=\Q^m\setminus\{0\}$.   The canonical basis of $\R^n$  is denoted by $e_1,\dots, e_n$. 
For $a, b \in \R^n$, $a \cdot b$ denotes the Euclidean scalar product between $a$ and $b$,
and $|\cdot|$ denotes the absolute value in $\R$ or the Euclidean norm in $\R^n$, $\R^m$, 
or $\R^{m{\times}n}$, depending on the context.

\item[(b)]$\Sph^{n-1}:=\{x=(x_1,\dots,x_n)\in \mathbb{R}^{n}: x_1^2+\dots+x_n^2=1\}$, $\Sph^{n-1}_\pm:=\{x\in \Sph^{n-1}:\pm x_n>0\}$, and $\widehat\Sph^{n-1}_\pm:=\{x\in \Sph^{n-1}: \pm x_{i(x)}>0\}$,
where $i(x)$ is the largest $i\in\{1,\dots,n\}$ such that $x_i\neq 0$. Note that $\Sph^{n-1}_\pm \subset \widehat\Sph^{n-1}_\pm$, and that $\Sph^{n-1} = \widehat\Sph^{n-1}_+ \cup \widehat\Sph^{n-1}_-$.
\item[(c)] $\mathcal{L}^n$ denotes the Lebesgue measure on $\R^n$ and $\mathcal{H}^{n-1}$ the $(n-1)$-dimensional Hausdorff measure on $\R^n$. 
\item[(d)] $\A$ denotes the collection of all bounded open subsets of $\R^n$; if $A$, $B\in \A$, by
$A\subset\subset B$ we mean that $A$ is relatively compact in $B$.
\smallskip
\item[(e)] For $u\in GSBV(A,\R^m)$ (see \cite[Section 4.5]{AFP}), with  $A \in \A$, the jump of $u$ across $S_u$ is defined by 
$[u]:=u^+-u^-$.
\item[(f)] For $A \in \A$ we define
$$
SBV_{\mathrm{pc}}(A,\R^m):=\{u\in SBV(A,\R^m): \nabla u=0 \,\, \mathcal{L}^{n}\textrm{-a.e.}, \,\mathcal{H}^{n-1}(S_u)<+\infty\}.
$$
\item[(g)] For $A \in \A$ and $p>1$ we define
$$
SBV^p(A,\R^m):=\{u\in SBV(A,\R^m): \nabla u \in L^p(A, \R^{m{\times}n}), \,\mathcal{H}^{n-1}(S_u)<+\infty\}.
$$
\item[(h)] For $A \in \A$ and $p>1$ we define
$$
GSBV^p(A,\R^m):=\{u\in GSBV(A,\R^m): \nabla u  \in L^p(A ,\R^{m{\times}n}), \,\mathcal{H}^{n-1}(S_u)<+\infty\};
$$
it is known that $GSBV^p(A,\R^m)$ is a vector space and that
for every $u\in GSBV^p(A,\R^m)$ and for every $\psi\in C^1_c(\R^m,\R^m)$
we have $\psi(u)\in SBV^p(A,\R^m)\cap L^\infty(A,\R^m)$  (see, e.g., \cite[page 172]{DMFT}).
\item[(i)] For every $\mathcal{L}^n$-measurable set $A \subset \R^n$ let 
$L^0(A,\R^m)$ be the space of all  ($\mathcal{L}^n$-equivalence classes of)  $\mathcal{L}^n$-measurable functions $u\colon A\to \R^m$,
endowed with the topology of convergence in measure on bounded subsets of $A$;
we observe that this topology is metrisable and separable.
\item[(j)] For $x\in \mathbb{R}^n$ and $\rho>0$ we define 
\begin{eqnarray*}
& B_\rho(x):=\{y\in \R^n: \,|y-x|<\rho\},
\\
& Q_\rho(x):= \{y\in \R^n: \,|(y-x)\cdot e_i | < \rho/2 \quad\hbox{for } i=1,\dots,n \}.
\end{eqnarray*}
We omit the subscript $\rho$ when $\rho=1$.

\item[(k)] For every $\nu\in \Sph^{n-1}$ let $R_\nu$ be an orthogonal $n{\times} n$ matrix such that $R_\nu e_n=\nu$; we assume that the restrictions of the function $\nu\mapsto R_\nu$ to the sets $\widehat\Sph^{n-1}_\pm$ defined in (b) are continuous and that $R_{-\nu}Q(0)=R_\nu Q(0)$ for every $\nu\in \Sph^{n-1}$;  moreover, we assume that 
$R_\nu \in O(n) \cap \mathbb{Q}^{n \times n}$ 
for every $\nu \in \mathbb{Q}^n \cap \Sph^{n-1}$.  
A map $\nu\mapsto R_\nu$ satisfying these properties is provided in \cite[Example A.1  and Remark A.2]{CDMSZ}.

\item[(l)] For $x\in \mathbb{R}^n$, $\rho>0$, and $\nu\in \Sph^{n-1}$ we set
$$
Q^\nu_\rho(x):= 
R_\nu Q_\rho(0) + x;
$$
we omit the subscript $\rho$ when $\rho=1$.
\item[(m)] For $\xi\in \R^{m{\times}n}$, the linear function from $\R^n$ to $\R^m$ with gradient $\xi$ is denoted by $\ell_\xi$\ie $\ell_\xi(x):=\xi x$, where $x$ is considered as an $n{\times}1$ matrix.
\item[(n)] For $x\in \mathbb{R}^n$, $\zeta\in \R^m_0$, and $\nu \in \Sph^{n-1}$ we define the function $u_{x,\zeta,\nu}$ as
\begin{equation*}
u_{x,\zeta,\nu}(y):=\begin{cases}
\zeta \quad & \mbox{if} \,\, (y-x)\cdot \nu \geq 0,\\
0 \quad & \mbox{if} \,\, (y-x)\cdot \nu < 0.
\end{cases}
\end{equation*}
\item[(o)] For  $x\in \mathbb{R}^n$ and $\nu \in \Sph^{n-1}$, we set
$$
\Pi^{\nu}_0:= \{y\in \R^n: y\cdot \nu = 0\}\quad\text{and}\quad\Pi^{\nu}_{x}:= \{y\in \R^n: (y-x)\cdot \nu = 0\}.
$$ 
\item[(p)] For a given topological space $X$, $\B(X)$ denotes the Borel $\sigma$-algebra on $X$. 
In particular, for every integer $ k \geq 1$, $\B^k$ is
the Borel $\sigma$-algebra on $\R^k$, while $\B^n_S$ 
stands for the Borel $\sigma$-algebra on $\mathbb{S}^{n-1}$. 
\item[(q)]  For every $t\in\R$  the integer part of $t$  is denoted by $\lfloor t\rfloor$\ie  $\lfloor t\rfloor$ is  the largest integer less than or equal to $t$. 

\end{itemize}

\section{Setting of the problem and statements of the main results}\label{S:setting}
This section consists of two parts: in Section \ref{Sub:setting} we introduce the stochastic free-discontinuity functionals and recall the Ergodic Subadditive Theorem; in Section \ref{sub:main_results} we state the main results of the paper.

\subsection{Setting of the problem}\label{Sub:setting}
\noindent Throughout the paper we fix six constants $p, c_1,\dots, c_5$, with  $1<p<+\infty$, $0<c_1\leq c_2<+\infty$, $1\le c_3<+\infty$, and  $0<c_4\leq c_5<+\infty$, and
two nondecreasing continuous functions $\sigma_1$, $\sigma_2\colon  [0,+\infty) \to [0,+\infty)$ such that $\sigma_1(0)=\sigma_2(0)=0$.

\begin{defn}[Volume and surface  integrands] Let $\mathcal{F}=\mathcal{F}(p,c_1,c_2,\sigma_1)$ be the collection of all functions $f\colon \R^n{\times} \R^{m{\times}n}\to [0,+\infty)$ satisfying the following conditions:
\begin{itemize}
\item[$(f1)$] (measurability) $f$ is Borel measurable on $\R^n{\times} \R^{m{\times}n}$;
\item[$(f2)$] (continuity in $\xi$) for every $x \in \R^n$ we have
\begin{equation*}
|f(x,\xi_1)-f(x,\xi_2)| \leq \sigma_1(|\xi_1-\xi_2|)\big(1+f(x,\xi_1)+f(x,\xi_2)\big)
\end{equation*}
for every $\xi_1$, $\xi_2 \in \R^{m{\times}n}$;
\item[$(f3)$] (lower bound) for every $x \in \R^n$ and every $\xi \in \R^{m{\times}n}$
$$
c_1 |\xi|^p \leq f(x,\xi);
$$
\item[$(f4)$] (upper bound) for every $x \in \R^n$ and every $\xi \in \R^{m{\times}n}$
$$
f(x,\xi) \leq c_2(1+|\xi|^p).
$$
\end{itemize}

Let $\mathcal{G}=\mathcal{G}(c_3, c_4,c_5, \sigma_2)$ be the collection of all functions 
$g\colon \R^n{\times}\R^m_0{\times} \Sph^{n-1} \to [0,+\infty)$ satisfying the following conditions:
\begin{itemize}
\item[$(g1)$] (measurability) $g$ is Borel measurable on $\R^n{\times}\R^m_0{\times} \Sph^{n-1}$;
\item[$(g2)$] (continuity in $\zeta$) for every $x\in \R^n$ and every $\nu \in \Sph^{n-1}$ we have
\begin{equation*}
|g(x,\zeta_2,\nu)-g(x,\zeta_1,\nu)|\leq \sigma_2(|\zeta_1-\zeta_2|)\big(g(x,\zeta_1,\nu)+g(x,\zeta_2,\nu)\big)
\end{equation*}
for every $\zeta_1$, $\zeta_2\in \R^m_0$;
\item[$(g3)$] (estimate for $|\zeta_1|\le|\zeta_2|$) for every $x\in \R^n$ and every $\nu \in \Sph^{n-1}$
we have
$$
g(x,\zeta_1,\nu) \leq c_3 \,g(x,\zeta_2,\nu)
$$
for every $\zeta_1$, $\zeta_2 \in \R^m_0$ with $|\zeta_1|\le |\zeta_2|$;
\item[$(g4)$] (estimate for $c_3|\zeta_1|\le|\zeta_2|$) for every $x\in \R^n$ and every $\nu \in \Sph^{n-1}$
we have
\begin{equation*}
g(x,\zeta_1,\nu) \leq \,g(x,\zeta_2,\nu)
\end{equation*}
for every $\zeta_1$, $\zeta_2\in \R^m_0$ with $c_3|\zeta_1|\le|\zeta_2|$;
\item[$(g5)$] (lower bound) for every $x\in \R^n$, $\zeta\in \R^m_0$, and $\nu \in \Sph^{n-1}$
\begin{equation*}
c_4 \leq g(x,\zeta,\nu);
\end{equation*}
\item[$(g6)$] (upper bound) for every $x\in \R^n$, $\zeta\in \R^m_0$, and $\nu \in \Sph^{n-1}$
\begin{equation*}
g(x,\zeta,\nu) \leq c_5 (1+|\zeta|);
\end{equation*}
\item[$(g7)$] (symmetry) for every $x\in \R^n$, $\zeta\in \R^m_0$, and $\nu \in \Sph^{n-1}$
$$g(x,\zeta,\nu) = g(x,-\zeta,-\nu).$$ 
\end{itemize}
\end{defn}

\begin{rem}
As observed in \cite[Remark~3.3]{CDMSZ}, assumptions $(g3)$ and $(g4)$
are strictly weaker than a monotonicity condition in $|\zeta|$.
Indeed, if $g\colon \R^n{\times}\R^m_0{\times} \Sph^{n-1} \to [0,+\infty)$
satisfies
$$
\zeta_1, \zeta_2 \in \R^m_0 \text{ with }
|\zeta_1|\le |\zeta_2| \quad \Longrightarrow \quad g(x,\zeta_1,\nu) \leq  g(x,\zeta_2,\nu)
$$
for every $x\in \R^n$ and every $\nu \in \Sph^{n-1}$, then $g$ satisfies $(g3)$ and $(g4)$.
On the other hand, $(g3)$ and $(g4)$ do not imply monotonicity in $|\zeta|$. 
\end{rem}

Given $f\in \mathcal{F}$ and $g\in \mathcal{G}$, we consider the integral functionals 
$F,\, G \colon L^0(\R^n,\R^m){\times} \A \longrightarrow [0,+\infty]$ defined as
\begin{align}
F(u,A)&:=
\begin{cases}
\ds\int_A f(x, \nabla u) \dx & \text{if}\; u|_A\in W^{1,p}(A,\R^m),\cr
+\infty & \text{otherwise in}\; L^0(\R^n,\R^m).
\end{cases} \label{Effe} \\
G(u,A) &:= 
\begin{cases}
\ds \int_{S_u\cap A}g(x,[u],\nu_u)d \mathcal{H}^{n-1} &\text{if} \; u|_A\in GSBV^p(A,\R^m),\cr
+\infty \quad & \mbox{otherwise in}\; L^0(\R^n,\R^m), \label{surf}
\end{cases}
\end{align}

\begin{rem}%\label{well defined}
Since $[u]$ is reversed when the orientation of $\nu_u$ is reversed, the functional $G$ is well defined thanks to $(g7)$.
\end{rem}

Let $A\in\A$. For $F$ as in \eqref{Effe}, and $w\in L^0(\R^n,\R^m)$ with $w|_A\in W^{1,p}(A,\R^m)$, we set
\begin{equation}\label{emme}
m^{1,p}_F(w,A):= \inf\left\{F(u,A): u \in L^0(\R^n,\R^m),
\ u|_A\in W^{1,p}(A,\R^m),\; u=w \textrm{ near } \partial A \right\}.
\end{equation}
Moreover, for $G$ as in \eqref{surf}, and  $w\in L^0(\R^n,\R^m)$ with $w|_A\in SBV_{\mathrm{pc}}(A,\R^m)$, we set
\begin{align}\label{emmeG}
m^{\mathrm{pc}}_G(w,A):= \inf\left\{G(u,A): u \in L^0(\R^n,\R^m),
\ u|_A\in SBV_{\mathrm{pc}}(A,\R^m),\ u=w \textrm{ near }\partial A \right\}.
\end{align}
In \eqref{emme} and \eqref{emmeG}, by ``$u=w \textrm{ near }\partial A$''
we mean that there exists a neighbourhood $U$ of $\partial A$ such that
$u=w$ $\mathcal{L}^n$-a.e.\ in $U$.

If $A$ is an arbitrary bounded subset of $\R^n$, we set $m^{1,p}_F(w,A):=m^{1,p}_F(w,\mathrm{int}A)$ and 
$m^{\mathrm{pc}}_G(w,A):=m^{\mathrm{pc}}_G(w,\mathrm{int}A)$, where $\mathrm{int}$ denotes the interior of $A$.

\begin{rem}\label{min-pc-bdd}
Let $u\in L^0(\R^n,\R^m)$ be such that  $u|_A\in SBV_{\mathrm{pc}}(A,\R^m)$, and let $k\in \N$. A careful inspection of the proof of \cite[Lemma 4.1]{CDMSZ} shows that there exist $\mu_k> k$ and $v_k\in L^\infty(\R^n,\R^m)$ with $v_k|_A \in SBV_{\mathrm{pc}}(A,\R^m)$ such that 
\begin{align*}
\|v_k\|_{L^\infty(A,\R^m)} \leq \mu_k, \quad
v_k=u\; \text{ $\mathcal L^n$-a.e. in }\; \{|u|\leq k\},\quad
G(v_k,A)\leq \Big(1+\frac{1}{k}\Big)G(u,A).
\end{align*}
As a consequence we may readily deduce the following. Let $w\in L^0(\R^n,\R^m)$ be such that $w|_A\in SBV_{\mathrm{pc}}(A,\R^m)\cap L^\infty(A,\R^m)$ and let $k\in\N$, $k >\|w\|_{L^\infty(A,\R^m)}$ be fixed. Then
\begin{equation}\label{approx-min-pb-G}
m^{\mathrm{pc}}_G(w,A)=\inf_{k} m^k_G(w,A)=\lim_{k\to +\infty}m^k_G(w,A),
\end{equation}
where
\begin{multline}\label{emme-kappa}
m^k_G(w,A):=\inf\big\{G(u,A)\colon u \in L^0(\R^n,\R^m),
\ u|_A\in SBV_{\mathrm{pc}}(A,\R^m) \cap L^\infty(A,\R^m), \\
\|u\|_{L^\infty(A,\R^m)} \leq k,\, \mathcal H^{n-1}(S_u\cap A) \leq \alpha,\  u=w \textrm{ near }\partial A \big\},
\end{multline}
with $\alpha:=c_5/c_4\, (1+2 \|w\|_{L^\infty(A,\R^m)}) \,\mathcal H^{n-1}(S_w\cap A)$. 
The fact that  the inequality  $\mathcal{H}^{n-1}(S_u\cap A) \leq \alpha$ in \eqref{emme-kappa} 
is not restrictive follows  from  assumption $(g6)$ by using $w$ as a competitor in the minimisation problem defining
$m^k_{G}(w,A)$  for $k>\|w\|_{L^\infty(A,\R^m)}$.

\end{rem}

\bigskip

We are now ready to introduce the probabilistic setting of our problem.
In what follows $(\Om,\T,P)$ denotes a fixed probability space.

\begin{defn}[Random integrand]\label{ri}
A function $f \colon \Om \times \R^n \times \R^{m \times n} \to [0,+\infty)$ is called a \textit{random volume integrand} if
\begin{itemize}
\item[(a)]  $f$ is $(\T\otimes \B^n\otimes \B^{m\times n})$-measurable; 
\item[(b)]  $f(\om,\cdot,\cdot)\in \mathcal{F}$ for every $\om\in \Om$. 
\end{itemize}
A function $g \colon \Om \times \R^n\times \R^m_0 \times \Sph^{n-1}  \to [0,+\infty)$ is called a \textit{random surface integrand} if
\begin{itemize}
\item[(c)]  $g$ is $(\T\otimes \B^n\otimes \B^m\otimes \B^{n}_S)$-measurable; 
\item[(d)] $g(\om,\cdot,\cdot,\cdot)\in \mathcal{G}$ for every $\om\in \Om$.
\end{itemize}
\end{defn}

Let $f$ be a random volume integrand. For $\om \in \Om$  the integral functional $F(\om) \colon L^{0}(\R^n,\R^m) \times \A \longrightarrow [0,+\infty]$ is defined by \eqref{Effe}, with $f(\cdot,\cdot)$ replaced by $f(\om,\cdot,\cdot)$.
Let $g$ be a random surface integrand. For $\om \in \Om$ the  integral functional $G(\om) \colon L^{0}(\R^n,\R^m) \times \A \longrightarrow [0,+\infty]$ is defined by \eqref{surf}, with
$g(\cdot,\cdot,\cdot)$ replaced by $g(\om,\cdot,\cdot,\cdot)$. Finally, for every $\e>0$ we consider the free-discontinuity functional $E_\e(\om)\colon L^{0}(\R^n,\R^m)\times \A \longrightarrow [0,+\infty]$ defined by
\begin{equation}\label{Eeps}
E_\e(\om)(u,A) :=\begin{cases}
\ds\int_A\!\! f\Big(\om,\frac{x}{\e}, \nabla u\Big) dx + \!\! \int_{S_u\cap A} \!\!\!\!\!  g\Big(\om,\frac{x}\e,[u],\nu_u\Big)d \mathcal{H}^{n-1}& \!\!\text{if} \; u|_A\!\in GSBV^p(A,\R^m),\cr
+\infty & \!\! \text{otherwise\! in}\,L^{0}(\R^n,\R^m).
\end{cases}
\end{equation}

In the study of stochastic homogenisation an important role is played by the notions introduced by the following definitions.

\begin{defn}[$P$-preserving transformation]
A $P$-preserving transformation on $(\Om,\T,P)$ is a map $T \colon \Om \to \Om$ satisfying the following properties:
\begin{itemize}
\item[(a)] (measurability) $T$ is $\T$-measurable;
\item[(b)] (bijectivity) $T$ is bijective;
\item[(c)]  (invariance) $P(T(E))=P(E)$, for every $E\in \T$.
\end{itemize}
\noindent If, in addition, every set $E\in \T$ which satisfies 
$T(E)=E$ (called $T$-invariant set) has probability $0$ or $1$, then $T$ is called {\it ergodic}.
\end{defn}

\begin{defn}[Group of $P$-preserving transformations]\label{P-preserving} Let $k$ be a positive integer.
A group of $P$-preserving transformations on $(\Om,\T,P)$ is a family $(\tau_z)_{z\in \Z^k}$ of mappings $\tau_z \colon \Om \to \Om$ satisfying the following properties:
\begin{itemize}
\item[(a)] (measurability) $\tau_z$ is $\T$-measurable for every $z\in \Z^k$;
\item[(b)] (bijectivity) $\tau_z$ is bijective for every $z\in \Z^k$;
\item[(c)]  (invariance) $P(\tau_z(E))=P(E)$,\, for every $E\in \T$ and every $z\in \Z^k$;
\item[(d)] (group property) $\tau_{0}={\mathrm id}_\Om$ (the identity map on $\Om$) and $\tau_{z+z'}=\tau_{z} \circ \tau_{z'}$ for every $z, z'\in \Z^k$;
\end{itemize}
\noindent If, in addition, every set $E\in \T$ which satisfies 
$\tau_z (E)=E$ for every $z\in \Z^k$
has probability $0$ or $1$, then $(\tau_z)_{z\in\Z^k}$ is called {\it ergodic}.
\end{defn}

\begin{rem}
In the case $k=1$ a group of $P$-preserving transformations has the form $(T^z)_{z\in \Z}$, where $T:=\tau_1$ is a $P$-preserving transformation.
\end{rem}

We are now in a position to define the notion of stationary random integrand.
\begin{defn}[Stationary random integrand]%\label{def:stationary}
A random volume integrand $f$ is \textit{stationary} with respect to a group $(\tau_z)_{z\in \Z^n}$ of $P$-preserving transformations on $(\Om,\T,P)$ if
$$
f(\om,x+z,\xi)=f(\tau_z(\om),x,\xi)
$$
for every $\om\in \Om$, $x\in \R^n$, $z\in \Z^n$, and $\xi\in \R^{m \times n}$.

Similarly, a random surface integrand $g$ is \textit{stationary} with respect to $(\tau_z)_{z\in \Z^{n}}$ if 
\begin{equation}\label{covariance}
g(\om,x+z,\zeta,\nu)=g(\tau_z(\om),x,\zeta,\nu)
\end{equation}
for every $\om\in \Om$, $x\in \R^n$, $z\in \Z^{n}$, $\zeta\in \R^m_0$, and $\nu \in \Sph^{n-1}$.
\end{defn} 

\medskip

We now recall the notion of subadditive stochastic processes as well as the Subadditive Ergodic Theorem by Akcoglu and Krengel \cite[Theorem 2.7]{AK}.  

\medskip

Let $k$ be a positive integer. For every $a,b\in \R^k$, with $a_i< b_i$ for $i=1,\dots,k$, we define
$$
\quad [a, b):= \{x\in \R^k: a_i\leq x_i<b_i\text{ for } i=1,\dots,k\},
$$
and we set 
\begin{equation}\label{int:int}
\mathcal{I}_k:= \{[a, b): a,b\in \R^k, a_i< b_i\text{ for }i=1,\dots,k\}.
\end{equation}
\medskip

\begin{defn}[Subadditive process]\label{Def:subadditive} A {\em subadditive process} with respect to a group $(\tau_z)_{z\in \Z^k}$, $k\ge 1$, of $P$-preserving transformations on $(\Om,\T,P)$ is  a function $\mu\colon\Om\times \mI_k\to \R$ satisfying the following properties: 
\begin{itemize}
\item[(a)] (measurability) for every $A\in \mI_k$ the function $\om\mapsto \mu(\om,A)$ is $\T$-measurable;
\item[(b)] (covariance) for every $\om\in\Om$, $A\in \mI_k$, and $z\in \Z^k$ we have
$
\mu(\om, A+z)=\mu(\tau_z(\om),A)$;
\item[(c)] (subadditivity) for every $A\in \mI_k$ and for every \emph{finite} family $(A_i)_{i\in I} \subset \mI_k$ of pairwise disjoint sets such that $A= \cup_{i\in I} A_i$, we have
$$
\mu(\om,A)\leq \sum_{i\in I} \mu(\om,A_i)\quad\hbox{for every } \om\in\Om; 
$$
 \item[(d)] (boundedness) there exists $c>0$ such that $0\leq \mu(\om, A) \leq c \mathcal L^k(A)$ for every $\om\in\Om$ and every $A\in \mI_k$. 
\end{itemize}
\end{defn}

We now state a variant of the pointwise ergodic Theorem \cite[Theorem 2.7 and Remark p. 59]{AK} which is suitable for our purposes (see, e.g., \cite[Proposition 1]{DMM2}). 

\begin{thm}[Subadditive Ergodic Theorem]\label{ergodic} Let $k\in \{n-1,n\}$ and let $(\tau_z)_{z\in \Z^k}$ be a group of $P$-preserving transformations on $(\Om,\T,P)$. Let $\mu \colon \Om\times \mI_k\to \R$ be a subadditive process  with respect to $(\tau_z)_{z\in \Z^k}$, and let $Q=Q_\rho(x)$ be an arbitrary cube with sides parallel to the coordinate axes.
Then there exist a $\T$-measurable function $\varphi \colon \Om \to [0,+\infty)$ and a set $\Om'\in\T$ with $P(\Om')$=1 such that
\begin{equation*}
\lim_{t \to +\infty} \frac{\mu(\om,tQ)}{\mathcal L^k(tQ)}=\varphi(\om), 
\end{equation*}
for every $\om\in \Om'$. If in addition $(\tau_z)_{z\in \Z^k}$ is ergodic, then $\varphi$ is constant $P$-a.e. 
\end{thm}

\subsection{Statement of the main results}\label{sub:main_results}
In this section we state the main result of the paper, Theorem \ref{G-convE}, which provides a $\Gamma$-convergence and integral representation result for the random functionals $(E_\e(\om))_{\e > 0}$ introduced in \eqref{Eeps}, under the assumption that the volume and surface integrands $f$ and $g$ are stationary. The volume and surface integrands of the $\Gamma$-limit are given in terms of \textit{separate}  asymptotic  cell formulas, showing that there is no interaction between volume and surface densities by stochastic $\Gamma$-convergence. 

 The next theorem proves the existence of the  limits in the asymptotic  cell formulas that will be used in the statement of the main result. The proof will be given in Sections \ref{Ergodic}-\ref{not special}.

\medskip

\begin{thm}[Homogenisation formulas]\label{en-density_vs}
Let $f$ be a stationary random volume integrand and let $g$ be a stationary random surface integrand  with respect to a group $(\tau_z)_{z\in \Z^n}$ of $P$-preserving transformations on $(\Om,\T,P)$. For every $\om\in\Om$ let $F(\om)$ and $G(\om)$ be defined by \eqref{Effe} and \eqref{surf}, with $f(\cdot,\cdot)$ and $g(\cdot,\cdot,\cdot)$
replaced by $f(\om,\cdot,\cdot)$ and $g(\om,\cdot,\cdot,\cdot)$, respectively. Finally, let $m^{1,p}_{F(\om)}$ and $m^{\mathrm{pc}}_{G(\om)}$ be defined by \eqref{emme} and \eqref{emmeG}, respectively.
Then there exists $\Om'\in \T$, with $P(\Om')=1$,
such that for every $\om\in \Om'$, $x\in \R^{n}$, $\xi \in \R^{m\times n}$, $\zeta\in \R^m_0$, and $\nu\in \Sph^{n-1}$ the limits  
\begin{align*}
\lim_{t\to +\infty} \frac{m^{1,p}_{F(\om)}(\ell_\xi,Q_t(tx))}{t^{n}} \quad \textrm{and} \quad \lim_{t\to +\infty} \frac{m^{\mathrm{pc}}_{G(\om)}(u_{t x,\zeta,\nu},Q^\nu_t(tx))}{t^{n-1}}
\end{align*}
exist and are independent of $x$. More precisely, there exist a random volume integrand $f_{\mathrm{hom}} \colon \Om\times \R^{m\times n} \to [0,+\infty)$, and a random surface integrand $g_{\mathrm{hom}} \colon \Om\times \R^m_0\times \Sph^{n-1} \to [0,+\infty)$ such that for every $\om\in\Om'$, $x \in \mathbb{R}^n$,
$\xi \in \mathbb{R}^{m\times n}, \zeta \in \mathbb{R}^m_0$, and $\nu \in \mathbb{S}^{n-1}$
\begin{align}
f_{\mathrm{hom}}(\om,\xi)=&\lim_{t\to +\infty} \frac{m^{1,p}_{F(\om)}(\ell_\xi,Q_t(tx))}{t^{n}}=  \lim_{t\to +\infty} \frac{m^{1,p}_{F(\om)} (\ell_\xi, Q_t(0))}{t^n},
\label{psi0}
\\
g_{\mathrm{hom}}(\om,\zeta,\nu)=&\lim_{t\to +\infty} \frac{m^{\mathrm{pc}}_{G(\om)}(u_{t x,\zeta,\nu},Q^\nu_t(tx))}{t^{n-1}}= \lim_{t\to +\infty} \frac{m^{\mathrm{pc}}_{G(\om)}(u_{0,\zeta,\nu},Q^\nu_t(0))}{t^{n-1}}.
\label{phi0}
\end{align}
If, in addition, $(\tau_z)_{z\in \Z^n}$ is ergodic, then $f_{\mathrm{hom}}$ and $g_{\mathrm{hom}}$ are independent of $\om$ and
\begin{align*}
f_{\mathrm{hom}}(\xi)&=\lim_{t\to +\infty}\, \frac{1}{t^n} {\int_\Om m^{1,p}_{F(\om)} (\ell_\xi, Q_t(0))\,dP(\om)}, 
\\
g_{\mathrm{hom}}(\zeta,\nu)&= \lim_{t\to +\infty} \frac{1}{t^{n-1}} \int_\Om m^{\mathrm{pc}}_{G(\om)} (u_{0,\zeta,\nu}, Q_t^\nu(0)) \, dP(\om).
\end{align*}
\end{thm}

We are now ready to state the main result of this paper, namely the almost sure $\Gamma$-convergence of the sequence of random functionals $(E_\e(\om))_{\e>0}$ introduced in \eqref{Eeps}.

\begin{thm}[$\Gamma$-convergence]\label{G-convE}
Let $f$ and $g$ be stationary random  volume and surface integrands with respect to a group $(\tau_z)_{z\in \Z^n}$ of $P$-preserving transformations on $(\Om,\T,P)$, let $E_\e(\om)$ be as in \eqref{Eeps}, let $\Om'\in \T$ (with $P(\Om')=1$), $f_{\mathrm{hom}}$, and $g_{\mathrm{hom}}$ be as in Theorem \ref{en-density_vs}, and let  $E_{\mathrm{hom}}(\om): L^0(\R^n, \R^m)\times \A  \longrightarrow [0,+\infty]$ be the free-discontinuity functional defined by
$$
E_{\mathrm{hom}}(\om)(u, A)= 
\begin{cases}
\ds\int_A f_{\mathrm{hom}}(\om,\nabla u)\dx + \int_{S_u\cap A}g_{\mathrm{hom}}(\om,[u],\nu_u)d \mathcal{H}^{n-1}&\text{if }u|_A\in GSBV^p(A,\R^m),
\\
+\infty&\text{otherwise in }L^0(\R^n, \R^m).
\end{cases}
$$ 
Let moreover $E^p_\e(\om)$ and $E^p_{\mathrm{hom}}(\om)$ be the restrictions to $L^p_{\mathrm{loc}}(\R^n,\R^m){\times} \A$ of $E_\e(\om)$ and $E_{\mathrm{hom}}(\om)$, respectively. Then 
$$
E_\e(\om)(\cdot,A) \; \Gamma\hbox{-converge to }E_{\mathrm{hom}}(\om) (\cdot,A)\hbox{ in }L^0(\R^n,\R^m),
$$
and
$$
E^p_\e(\om)(\cdot,A)\; \Gamma\hbox{-converge to }E^p_{\mathrm{hom}}(\om) (\cdot,A)\hbox{ in }L^p_{\mathrm{loc}}(\R^n,\R^m),
$$
for every $\om\in \Om'$ and every $A\in \A$.

Further, if $(\tau_z)_{z\in \Z^n}$ is ergodic, then $E_{\mathrm{hom}}$ (resp.~$E^p_{\mathrm{hom}}$) is a deterministic functional\ie it does not depend on $\om$.
\end{thm}

\begin{proof}
Let $\Om'\in \mathcal T$ be the the set with $P(\Om')=1$ whose existence is established in Theorem~\ref{en-density_vs} and let $\om\in \Om'$ be fixed. Then, the functionals $F(\om)$ and $G(\om)$ defined by \eqref{Effe} and \eqref{surf}, respectively (with $f$ replaced by $f(\om,\cdot,\cdot)$ and $g$ replaced by $g(\om,\cdot,\cdot,\cdot,\cdot)$)  satisfy all the assumptions of \cite[Theorem 3.7]{CDMSZ}. Therefore, by combining Theorem~\ref{en-density_vs} and \cite[Theorem 3.7]{CDMSZ} the conclusion follows.
\end{proof}
Thanks to Theorem \ref{G-convE} we can also characterise the asymptotic behaviour of some minimisation problems 
involving $E_\e(\om)$. An example is shown in the corollary below.

\begin{cor}[Convergence of minimisation problems]%\label{corollary} 
Let $f$ and $g$ be stationary random  volume and surface 
 integrands with respect to a group $(\tau_z)_{z\in \Z^n}$ of $P$-preserving transformations on $(\Om,\T,P)$, let $\Om'\in \T$ (with $P(\Om')=1$), $f_{\mathrm{hom}}$, and $g_{\mathrm{hom}}$ be as in Theorem \ref{en-density_vs}.
Let $\om\in\Om'$, $A\in \A$, $h\in L^p(A,\R^m)$, and let $(u_\e)_{\e > 0} \subset GSBV^p(A,\R^m)\cap L^p(A,\R^m)$ be a sequence such that
\begin{eqnarray*}
&\ds \int_A f\Big(\om,\frac{x}{\e}, \nabla u_\e\Big)dx + \int_{S_{u_\e}\cap A} g\Big(\om,\frac{x}\e,[u_\e],\nu_{u_\e}\Big)d \mathcal{H}^{n-1}+ \int_A|u_\e-h|^pdx
\\
&\ds\le \inf_{u} \bigg(\int_A f\Big(\om,\frac{x}{\e}, \nabla u\Big)dx + \int_{S_u\cap A} g\Big(\om,\frac{x}\e,[u],\nu_u\Big)d \mathcal{H}^{n-1}+ \int_A|u-h|^pdx\bigg)+\eta_\e
\end{eqnarray*}
for some $\eta_\e\to 0+$, where the infimum is taken over all $u\in GSBV^p(A,\R^m)\cap L^p(A,\R^m)$. Then there exists a sequence $\e_k\to 0+$ such that $(u_{\e_k})_{k \in \mathbb{N}}$ converges in $L^p(A,\R^m)$ to a minimiser $u_0$ of
$$
\int_A f_{\mathrm{hom}}(\om,\nabla u)\dx + \int_{S_u\cap A}g_{\mathrm{hom}}(\om,[u],\nu_u)d \mathcal{H}^{n-1}+ \int_A|u-h|^pdx
$$
on $GSBV^p(A,\R^m)\cap L^p(A,\R^m)$. Moreover
$$
\int_A f\Big(\om,\frac{x}{\e}, \nabla u_\e\Big)dx + \int_{S_{u_\e}\cap A} g\Big(\om,\frac{x}\e,[u_\e],\nu_{u_\e}\Big)d \mathcal{H}^{n-1}+ \int_A|u_\e-h|^p\,dx
$$
converges to
$$
\int_A f_{\mathrm{hom}}(\om,\nabla u_0)\dx + \int_{S_{u_0}\cap A}g_{\mathrm{hom}}(\om,[u_0],\nu_{u_0})d \mathcal{H}^{n-1}+ \int_A|u_0-h|^pdx
$$
as $\e\to 0+$.
\end{cor}
\begin{proof}
The proof follows  from  Theorem \ref{G-convE}, arguing as in the proof of \cite[Corollary 6.1]{CDMSZ}. 
\end{proof}

%%%%%%%%

\section{Proof  of the cell-formula for the volume integrand}

In this section we prove   \eqref{psi0}\label{Ergodic}.

\begin{prop}[Homogenised volume integrand] \label{en-density_vv}
Let $f$ be a stationary random volume integrand with respect to a group $(\tau_z)_{z\in \Z^n}$ of $P$-preserving transformations on $(\Om,\T,P)$. Then there exists $\Om'\in \T$, with $P(\Om')=1$,
such that for every $\om\in \Om'$, for every $x\in \R^n$, and $\xi \in \R^{m\times n}$ the limit
\begin{align*}
\lim_{t\to +\infty} \frac{m^{1,p}_{F(\om)}(\ell_\xi,Q_t(t x))}{t^{n}}
\end{align*}
exists and is independent of $x$. More precisely, there exists a random volume integrand $f_{\mathrm{hom}} \colon \Om\times \R^{m\times n} \to [0,+\infty)$, independent of $x$, such that 
\begin{align*}
f_{\mathrm{hom}}(\om,\xi)=\lim_{t\to +\infty} \frac{m^{1,p}_{F(\om)}(\ell_\xi,Q_t(t x))}{t^{n}}=  \lim_{t\to +\infty} \frac{m^{1,p}_{F(\om)} (\ell_\xi, Q_t(0))}{t^n}.
\end{align*}
If, in addition, $(\tau_z)_{z\in \Z^n}$ is ergodic, then $f_{\mathrm{hom}}$ is independent of $\om$ and
\begin{align*}
f_{\mathrm{hom}}(\xi)=\lim_{t\to +\infty}\, \frac{1}{t^n} {\int_\Om m^{1,p}_{F(\om)} (\ell_\xi, Q_t(0))\,dP(\om)} =\inf_{k\in \mathbb N}\, \frac{1}{k^n} {\int_\Om m^{1,p}_{F(\om)} (\ell_\xi, Q_k(0))\,dP(\om)}.
\end{align*}
\end{prop}
The proof of Proposition~\ref{en-density_vv} relies on the application of the  Subadditive  Ergodic Theorem~\ref{ergodic} 
to the function $ (\om,A)\mapsto  m^{1,p}_{F(\om)}(\ell_\xi,A)$, which is a subadditive process as shown below. 
\begin{prop}\label{propr_m}
Let $f$ be a stationary random volume integrand with respect to a group $(\tau_z)_{z\in \Z^n}$ of $P$-preserving transformations on $(\Om,\T,P)$ and let $F(\om)$ be as in \eqref{Effe} with $f(\cdot,\cdot)$ replaced by $f(\om,\cdot,\cdot)$. Let $\xi\in \R^{m\times n}$ and set 
$$
\mu_\xi(\om,A):=m^{1,p}_{F(\om)}(\ell_\xi, A)\quad  \text{for every } \om \in \Om,\ A\in \mI_n, 
$$ 
where $m^{1,p}_{F(\om)}$ is as in \eqref{emme} and $\mI_n$ as in \eqref{int:int}.
Then  $\mu_\xi$ is a subadditive process with respect to $(\tau_z)_{z\in \Z^n}$ and 
\begin{equation*}%\label{stimaL1}
0\leq \mu_\xi(\om, A) \leq c_2 (1+|\xi|^p) \mathcal{L}^n(A) \quad \textrm{for } \textrm{ every } \om\in \Om.
\end{equation*}
\end{prop}

\begin{proof}See \cite{DMM2} and also \cite[Proposition 3.2]{MeMi}.
\end{proof}
We can now give the proof of Proposition~\ref{en-density_vv}.
\begin{proof}[Proof of Proposition~\ref{en-density_vv}]
The existence of $f_{\rm{hom}}$ and its independence of $x$ follow from Proposition~\ref{propr_m}
and \cite[Theorem 1]{DMM2} (see also \cite[Corollary 3.3]{MeMi}).
The fact that $f_{\rm{hom}}$ is a random volume integrand can be shown arguing 
as in \cite[Lemma A.5 and Lemma A.6]{CDMSZ}, and this concludes the proof.   
\end{proof}

\section{Proof  of the cell-formula for the surface integrand: a special case }\label{existence ghom}

This section is devoted to the proof of  \eqref{phi0} in the the special case $x=0$. Namely, we prove 
the following result.

\begin{thm}\label{Ergodic-thm-surf}
Let $g$ be a stationary random surface integrand
with respect to a group $(\tau_z)_{z\in \Z^n}$ of $P$-preserving transformations on $(\Om,\T,P)$. 
Then there exist $\widetilde\Om\in\T$, with $P(\widetilde\Om)=1$, and a random surface  integrand   $g_{\mathrm{hom}}\colon \Om\times \R^m_0 \times \Sph^{n-1} \to \R$ such that
\begin{equation}\label{lim:density}
g_{\mathrm{hom}}(\omega,\zeta,\nu)=
\lim_{t\to +\infty} \frac{m^{\mathrm{pc}}_{G(\om)}(u_{0,\zeta,\nu},Q^\nu_t(0))}{t^{n-1}},
\end{equation} 
for every $\om\in \widetilde\Om$, $\zeta\in \R^m_0$, and $\nu\in \Sph^{n-1}$.
\end{thm}  
The proof of Theorem~\ref{Ergodic-thm-surf} will need several preliminary results.
A key ingredient will be the application 
of the Ergodic Theorem~\ref{ergodic} with $k=n-1$.
This is a nontrivial task, since it requires to define an $(n-1)$-dimensional subadditive process
starting from the $n$-dimensional set function $ A\mapsto  m^{\mathrm{pc}}_{G(\om)}(u_{0,\zeta,\nu}, A )$.
To this end, we are now going to illustrate a systematic way to transform $(n-1)$-dimensional intervals (see \eqref{int:int}) into $n$-dimensional intervals oriented along  a prescribed  direction 
$\nu \in \mathbb{S}^{n-1}$. 

For every $\nu \in \mathbb S^{n-1}$ let $R_\nu$ be the orthogonal $n \times n$ 
matrix  defined in point (k) of Section~\ref{Notation} (see also \cite[Example A.1]{CDMSZ}).
Then, the following properties are satisfied:
\begin{itemize}
\item $R_\nu e_n=\nu$ for every $\nu \in \mathbb S^{n-1}$; 
\item the restrictions of the function $\nu\mapsto R_\nu$ to the sets $\widehat\Sph^{n-1}_\pm$ are continuous;
\item $R_{-\nu}Q(0)=R_\nu Q(0)$ for every $\nu \in \mathbb S^{n-1}$.
\end{itemize}
 Moreover,  $R_\nu \in O(n) \cap \mathbb{Q}^{n \times n}$ 
for every $\nu \in \mathbb{Q}^n \cap \Sph^{n-1}$.
Since $R_\nu e_n=\nu$, we have that $\{R_\nu e_j\}_{j=1,\dots,n-1}$ 
is an orthonormal basis of $\Pi^\nu_0$. 
Let now $M_\nu$ be a positive integer such that $M_\nu R_\nu \in \mathbb{Z}^{n\times n}$.
Note that, in particular, for every $z'\in \mathbb{Z}^{n-1}$ we have that $M_\nu R_\nu(z',0)\in \Pi^\nu_0\cap \mathbb{Z}^n$, namely $M_\nu R_\nu$ maps integer vectors perpendicular to $e_n$ into integer vectors perpendicular to $\nu$.

Let $A'\in \mathcal I_{n-1}$; we define the (rotated) $n$-dimensional interval $T_\nu(A')$ as 
\begin{equation}\label{An}
T_\nu(A'):= M_\nu R_\nu \left(A'\times[-c, c)\right), \quad c:=\frac12 \max_{1\leq j\leq n-1}(b_j-a_j), 
\quad (M_\nu R_\nu \in \mathbb{Z}^{n\times n}).
\end{equation}
 The  next proposition is the analogue of Proposition \ref{propr_m} for the surface energy, 
and will be crucial in the proof Theorem~\ref{Ergodic-thm-surf}.

\begin{prop}\label{propr_ms} 
Let $g$ be a stationary surface integrand 
with respect to a group $(\tau_z)_{z\in \Z^n}$ of $P$-preserving transformations on $(\Om,\T,P)$, let $G(\om)$ be as in \eqref{surf}, with $g(\cdot,\cdot,\cdot)$ replaced by $g(\om,\cdot,\cdot,\cdot)$, let $\zeta\in \Q^m_0$,  and let $\nu \in \mathbb{Q}^n \cap \Sph^{n-1}$. For every $A'\in\mI_{n-1}$ and $\om\in \Om$ set 
\begin{equation}\label{def:tildemu}
\mu_{\zeta,\nu}(\omega,A'):=\frac{1}{M_\nu^{n-1}}\,m^{\mathrm{pc}}_{G(\om)}(u_{0,\zeta,\nu}, T_\nu(A')),
\end{equation}
where $m^{\mathrm{pc}}_{G(\om)}$ is as in \eqref{emmeG}, while $M_\nu$ and $T_\nu(A')$ are as in \eqref{An}. 
Let $(\Om,\widehat\T,\widehat P)$ denote the completion of the probability space $(\Om,\T, P)$.
Then there exists a group $(\tau_{z'}^{\nu})_{{z'}\in \Z^{n-1}}$ of $\widehat P$-preserving transformations
on $(\Om,\widehat\T,\widehat P)$
such that $\mu_{\zeta,\nu}$ is a subadditive process on $(\Om,\widehat\T,\widehat P)$ with respect to $(\tau_{z'}^{\nu})_{{z'}\in \Z^{n-1}}$. Moreover
\begin{equation}\label{stimatildemu}
0\leq \mu_{\zeta,\nu}(\om)(A') \leq c_4 (1+|\zeta|) \mathcal{L}^{n-1}(A') \qquad \text{ for $\widehat P$-a.e. } \om\in \Om.
\end{equation}
\end{prop}
\begin{proof}%[Proof of Proposition \ref{propr_ms}]
The $\widehat\T$-measurability of the function $\om \mapsto \mu_{\zeta,\nu}(\om,A')$ follows from the $\widehat\T$-measurability of $\om \mapsto m^{\mathrm{pc}}_{G(\om)}( u_{0,\zeta,\nu}, A)$ for every $A\in\A$. This is a delicate issue, which will be postponed to the Appendix.

Let now $\zeta\in \Q^m_0$,  and let $\nu \in \mathbb{Q}^n \cap \Sph^{n-1}$. 
By Proposition \ref{measurability}, for every $A'\in \mathcal I_{n-1}$ 
the function $\om \mapsto \mu_{\zeta,\nu}(\omega,A')$ is $\widehat\T$-measurable.
We are now going to prove that there exists a group 
$(\tau_{z'}^{\nu})_{{z'}\in \Z^{n-1}}$ of $\widehat P$-preserving transformations on $(\Om,\widehat\T,\widehat P)$ such that 
$$
\mu_{\zeta,\nu}(\om, A'+{z'})=\mu_{\zeta,\nu}(\tau_{z'}^{\nu} (\om),A'),\quad \text{for every } \om \in \Omega,\ z'\in \Z^{n-1},\text{and } A'\in \mathcal{I}_{n-1}.
$$
To this end fix ${z'}\in \mathbb{Z}^{n-1}$ and $A'\in \mathcal{I}_{n-1}$. Note that, by \eqref{An},
\begin{align*}
T_\nu(A'+{z'}) &= M_\nu R_{\nu}((A'+{z'})\times [-c, c)) = M_\nu R_\nu \left(A'\times \left[-c, c\right)\right) + M_\nu R_\nu({z'},0)= T_\nu(A') +z'_\nu,
\end{align*}
where $z'_\nu:= M_\nu R_\nu({z'},0)\in \Z^n$. 
Then, by \eqref{def:tildemu}
\begin{equation}\label{eq:mu zeta nu}
\mu_{\zeta,\nu} (\omega,A'+{z'}) = \frac{1}{M_\nu^{n-1}} m^{\mathrm{pc}}_{G(\omega)}(u_{0,\zeta,\nu},%\mathrm{int} 
T_\nu(A'+{z'}))= \frac{1}{M_\nu^{n-1}} m^{\mathrm{pc}}_{G(\omega)}(u_{0,\zeta,\nu},%\mathrm{int} 
T_\nu(A')+z'_\nu).
\end{equation}
Given $u\in L^0(\R^n,\R^m)$, let $v\in L^0(\R^n,\R^m)$ be defined by $v(x):=u(x+z'_\nu)$ for every $x\in \R^n$.
By a change of variables we have
\begin{equation*}
\int_{S_u\cap (T_\nu(A')+z'_\nu)} g(\omega,x,[u],\nu_u)d\hs^{n-1} (x)= \int_{S_v\cap T_\nu(A')} g(\omega,y+z'_\nu,[v],\nu_v)d\hs^{n-1} (y).
\end{equation*}
Since $z'_\nu\in \Z^{n}$, by the stationarity of $g$ we have also
$g(\omega,y+z'_\nu, [v], \nu_v)
=g(\tau_{z'_\nu}(\omega),y, [v], \nu_v)$.
From these equalities we obtain
\begin{equation}\label{eq:station:g}
G(\omega)(u,\mathrm{int}T_\nu(A')+z'_\nu)=G(\tau_{z'_\nu}(\omega))(v, \mathrm{int} T_\nu(A')).
\end{equation}
Since $z'_\nu$ is perpendicular to $\nu$, we have $u_{0,\zeta,\nu}(x)=u_{0,\zeta,\nu}(x+ z'_\nu)$ for every $x\in \R^n$.
Therefore, from  \eqref{emmeG}, \eqref{eq:mu zeta nu}, and \eqref{eq:station:g} we obtain that
$
\mu_{\zeta,\nu} (\omega, A'+{z'}) = \mu_{\zeta,\nu} (\tau_{z'_\nu}(\omega), A')
$.  
Thus, $\mu_{\zeta,\nu}$ is covariant with respect to the group $\big(\tau_{z'}^{\nu}\big)_{{z'}\in \Z^{n-1}}$
of $ \widehat P$-preserving transformations on $(\Om,\widehat\T,\widehat P)$
defined by 
$$
\big(\tau_{z'}^{\nu}\big)_{{z'}\in \Z^{n-1}}:=(\tau_{z'_\nu})_{{z'}\in \Z^{n-1}}.
$$
We now show that $\mu_{\zeta,\nu}$ is subadditive. To this end let $A' \in\mathcal{I}_{n-1}$ and let $(A'_i)_{1\le i\le N} \subset  \mathcal{I}_{n-1}$ be a finite family of pairwise disjoint sets such that $A' =  \bigcup_{i} A'_i$. For fixed $\eta>0$ and $i=1,\dots, N$, let  $u_i\in SBV_{\mathrm{pc}}( \mathrm{int} T_\nu(A'_i))$ be such that $u_i=u_{0,\zeta,\nu}$ in a neighbourhood of $\partial T_\nu(A'_i)$ and
\begin{eqnarray}\label{q:min:Gi}
G(\om)(u_i, \mathrm{int} T_\nu(A_i)) \leq m^{\mathrm{pc}}_{G(\om)}(u_{0,\zeta,\nu},  T_\nu(A'_i))+ \eta.
\end{eqnarray}
Note that $T_\nu(A')$ can differ from $\bigcup_{i}T_\nu(A'_i)$ but, by construction, 
we always have $\bigcup_{i} T_\nu(A'_i) \subset T_\nu(A')$. Now we define
$$
u(y):=\begin{cases}
u_i(y) & \text{if}\; y\in T_\nu(A'_i),\ i=1,\dots,N,
\cr
u_{0,\zeta,\nu}(y) & \text{if}\; y\in T_\nu(A') \setminus \bigcup_{i}T_\nu(A'_i);
\end{cases}
$$
then $u\in SBV_{\mathrm{pc}}( \mathrm{int}  T_\nu(A'))$ and $u=u_{0,\zeta,\nu}$ in a neighbourhood of $\partial T_\nu(A')$. Moreover, by the additivity and the locality of $G(\omega)$  we have
\begin{equation}\label{eq:additiveG}
G(\om)(u, \mathrm{int} T_\nu(A'))=\sum_{i=1}^N G(\om)(u_i, \mathrm{int} T_\nu(A'_i)) 
+ {\textstyle G(\om)\big(u_{0,\zeta,\nu}, \mathrm{int} ( T_\nu(A') \setminus \bigcup_{i}T_\nu(A'_i) ) \big)},
\end{equation}
where we have also used the fact that $S_u\cap \partial T_\nu(A'_i) = \emptyset$ for every $i=1,\dots,N$.
Note that the last term in \eqref{eq:additiveG} is equal to zero because the jump set of $u_{0,\zeta,\nu}$ is the hyperplane $\Pi^\nu_0$, which does not intersect $T_\nu(A') \setminus \bigcup_{i}T_\nu(A'_i)$; therefore
$$
G(\om)(u,\mathrm{int} T_\nu(A')) = \sum_{i=1}^N G(\om)(u_i,\mathrm{int} T_\nu(A'_i)). 
$$
As a consequence, by \eqref{q:min:Gi},
$$
m^{\mathrm{pc}}_{G(\om)}(u_{0,\zeta,\nu}, T_\nu(A'))
\leq \sum_{i=1}^N m^{\mathrm{pc}}_{G(\om)}(u_{0,\zeta,\nu},  T_\nu(A'_i))+ N \eta,
$$
thus the subadditivity of $\mu_{\zeta,\nu}$ follows from \eqref{def:tildemu} and from the arbitrariness of $\eta$.
 
Finally, in view of $(g6)$ for every $A'\in \mathcal I_{n-1}$ we have  
\begin{align*}
\mu_{\zeta,\nu}(\omega,A') &\leq \frac{1}{M_\nu^{n-1}}G(\om)(u_{0,\zeta,\nu},\mathrm{int} T_\nu(A')) 
\\
&\leq \frac{c_5 (1+|\zeta|)}{M_\nu^{n-1}} \hs^{n-1}(\Pi^\nu_0\cap T_\nu(A'))
\\
& =  c_5 (1+|\zeta|)\mathcal{L}^{n-1}(A'),
\end{align*}
and thus \eqref{stimatildemu}.
\end{proof}
To conclude the proof of Proposition \ref{Ergodic-thm-surf} we need two preliminary lemmas.
\begin{lem}\label{continuity-g-bar} 
Let $g\in\G$, let $G$ be as in \eqref{surf}, and let $m^{\mathrm{pc}}_G$ be as in \eqref{emmeG}. Let \underline{g}, $\overline{g} \colon \R^m_0\times \Sph^{n-1}\to [-\infty,+\infty]$  
be the functions defined by
\begin{equation*}%\label{C:liminf-1}
\underline{g}(\zeta,\nu):=\liminf_{t\to +\infty}\frac{m^{\mathrm{pc}}_{G}(u_{0,\zeta,\nu},Q_{t}^\nu(0))}{t^{n-1}}
\quad\text{and}\quad
%\end{equation*}
%and
%\begin{equation*}%\label{C:limsup-1}
\overline{g}(\zeta,\nu):=\limsup_{t\to +\infty}\frac{m^{\mathrm{pc}}_{G}(u_{0,\zeta,\nu},Q_{t}^\nu(0))}{t^{n-1}}.
\end{equation*}
Then $\underline{g}, \overline{g}\in \G$. 
\end{lem}

\begin{proof}
 It is enough to adapt the proof of \cite[Lemma A.7]{CDMSZ}.
\end{proof}

We will also need the following result.

\begin{lem}\label{continuity_in_zeta-nu} Let $g\in\G$, let $G$ be as in \eqref{surf}, and let $m^{\mathrm{pc}}_G$ be as in \eqref{emmeG}. Let $\undertilde g, \widetilde g\colon \R^n\times \R^m_0\times \Sph^{n-1}\to [-\infty,+\infty]$ be the functions defined by
\begin{equation}\label{C:g-utilde}
\undertilde{g}(x,\zeta,\nu):=\liminf_{t\to +\infty}\frac{m^{\mathrm{pc}}_{G}(u_{tx,\zeta,\nu},Q_{t}^\nu(tx))}{t^{n-1}}
\end{equation}
and
\begin{equation}\label{C:g-tilde}
\widetilde{g}(x,\zeta,\nu):=\limsup_{t\to +\infty}\frac{m^{\mathrm{pc}}_{G}(u_{tx,\zeta,\nu},Q_{t}^\nu(tx))}{t^{n-1}}.
\end{equation}
Then $\undertilde{g}$ and $\widetilde{g}$ satisfy $(g2)$. Moreover for every $x\in \R^n$ and $\zeta\in\R^m_0$ the restriction of the functions $\nu\mapsto \undertilde{g}(x,\zeta,\nu)$ and $\nu\mapsto \widetilde{g}(x,\zeta,\nu)$ to the sets $\widehat\Sph^{n-1}_+$ and $\widehat\Sph^{n-1}_-$ are continuous.
\end{lem}

\begin{proof}
 The proof of $(g2)$ can be obtained by adapting the proof of  \cite[Lemma A.7]{CDMSZ}. 

 To  prove the continuity of $\nu\mapsto \undertilde{g}(x,\zeta,\nu)$ on $\widehat\Sph^{n-1}_+$, we fix   
$x\in\R^n$, $\zeta\in\R^m_0$, $\nu\in \widehat\Sph^{n-1}_+$, and   a sequence  $(\nu_j) \subset \widehat\Sph^{n-1}_+$  such that $\nu_j \to \nu$ as $j\to +\infty$. Since the function $\nu\mapsto R_\nu$ is continuous on 
$\widehat\Sph^{n-1}_+$, for every $\delta\in(0,\frac12)$ there exists an integer $j_\delta$ such that
\begin{equation}\label{cubes}
Q^{\nu_j}_{(1-\delta)t}(tx)  \subset\subset  Q^\nu_{t}(tx)  \subset\subset  Q^{\nu_j}_{(1+\delta)t}(tx),
\end{equation}
for every $j\geq j_\delta$ and every $t>0$.
Fix $j\geq j_\delta$, $t>0$, and $\eta>0$. Let $u\in SBV_{\mathrm{pc}}(Q^\nu_{t}(tx),\R^m)$ be such that $u=u_{tx,\zeta,\nu}$ in a neighbourhood of $\partial Q^\nu_{t}(tx)$, and
$$
G(u,Q^\nu_{t}(tx)) \leq m^{\mathrm{pc}}_{G}(u_{tx,\zeta,\nu}, Q^\nu_{t}(tx))+\eta.
$$
We set
$$
v(y) := 
\begin{cases}
u(y) \,\, &\textrm{if} \; y\in Q^\nu_{t}(tx),\\
u_{tx,\zeta,\nu_j}(y) \,\, &\textrm{if} \; y\in Q^{\nu_j}_{(1+\delta)t}(tx)\setminus Q^\nu_{t}(tx).
\end{cases} 
$$
Then $v\in SBV_{\mathrm{pc}}(Q^{\nu_j}_{(1+\delta)t}(tx),\R^m)$, $v=u_{tx,\zeta,\nu}$ in a neighbourhood of $\partial Q^{\nu_j}_{(1+\delta)t}(tx)$, and $S_{v}\subset S_{u}\cup\Sigma$, where 
$$
 \Sigma:=\big\{y\in \partial Q^\nu_{t}(tx): \big((y-tx){\,\cdot\,}\nu\big)\big((y-tx){\,\cdot\,}\nu_j\big)<0 \big\}  \cup \Pi^{\nu_j}_{tx}\cap(Q^{\nu_j}_{(1+\delta)t}(tx)\setminus Q^\nu_{t}(tx)).
$$%
By \eqref{cubes} there exists $\varsigma(\delta)>0$, independent of $j$ and $t$, with $\varsigma(\delta)\to 0$  as $\delta \to 0+$, such that $\hs^{n-1}(\Sigma)\le  \varsigma(\delta)t^{n-1}$. Thanks to $(g6)$ we then have 
\begin{align*}
m^{\mathrm{pc}}_{G}(u_{tx,\zeta,\nu_j}, Q^{\nu_j}_{(1+\delta)t}(tx))&\le 
G(v,Q_{(1+\delta)t}^{\nu_j}(tx)) 
\le G(u,Q^\nu_{t}(tx)) + \varsigma(\delta)c_5 (1+|\zeta|) t^{n-1}
\\
&\leq m^{\mathrm{pc}}_{G}(u_{tx,\zeta,\nu}, Q^\nu_{t}(tx))+ \eta + \varsigma(\delta)c_5 (1+|\zeta|) t^{n-1}.
\end{align*}
By dividing the terms of the above estimate by $t^{n-1}$ and passing to the liminf as $t\to +\infty$,  from 
\eqref{C:g-utilde} we obtain that
\begin{align*}
\undertilde{g}(x,\zeta,\nu_j)(1+\delta)^{n-1} \leq \undertilde{g}(x,\zeta,\nu)+  \varsigma(\delta)c_5(1+|\zeta|).
\end{align*}
Letting $j\to +\infty$ and then $\delta \to 0+$  we  deduce that 
\begin{equation*}\label{d:1}
\limsup_{j \to +\infty} \undertilde{g}(x,\zeta,\nu_j) \leq \undertilde{g}(\omega,\zeta,\nu).
\end{equation*}
An analogous argument,  now using the cube $Q^{\nu_j}_{(1-\delta)t}(tx)$,  yields  
\begin{equation*}\label{d:2}
\undertilde{g}(x,\zeta,\nu) \leq \liminf_{j \to +\infty} \undertilde{g}(\omega,\zeta,\nu_j),
\end{equation*}
and hence the continuity of $\undertilde{g}(x,\zeta,\cdot)$ in $\widehat\Sph^{n-1}_+$.  The proof of the continuity in $\widehat\Sph^{n-1}_-$, as well as that of the continuity of $\widetilde g$ are similar.
\end{proof}

We are now ready to prove Theorem \ref{Ergodic-thm-surf}. 

\begin{proof}[Proof of Theorem \ref{Ergodic-thm-surf}] Let  $(\Om,\widehat\T,\widehat P)$ be the completion of the probability space $(\Om,\T, P)$. For $\zeta\in \Q^m_0$ and $\nu\in \mathbb{Q}^n\cap \Sph^{n-1}$ fixed we apply the  Subadditive Ergodic Theorem \ref{ergodic} to the subadditive process $\mu_{\zeta,\nu}$ defined on $(\Om,\widehat\T,\widehat P)$ by \eqref{def:tildemu}. 
Choosing $Q':=[-\frac12,\frac12)^{n-1}$, we obtain the existence of a set $ \widehat \Omega_{\zeta,\nu}\in \widehat \T$, with $\widehat P(\widehat \Omega_{\zeta,\nu})=1$, and of a $\widehat\T$-measurable function $g_{\zeta,\nu}\colon \Omega\to \R$ such that 
\begin{equation} \label{lim mu zeta nu}
\lim_{t\to +\infty} \frac{\mu_{\zeta,\nu}(\omega)(tQ')}{t^{n-1}} = g_{\zeta,\nu}(\omega)
\end{equation}
for every $\omega\in \widehat \Omega_{\zeta,\nu}$. Then, by the properties of the completion there exist a set 
$\Omega_{\zeta,\nu}\in \T$, with $P(\Omega_{\zeta,\nu})=1$, and a $\T$-measurable function which we still denote by $g_{\zeta,\nu}$ such that 
\eqref{lim mu zeta nu} holds
for every $\omega\in \Omega_{\zeta,\nu}$.
Using the definition of $\mu_{\zeta,\nu}$ we then have 
\begin{equation*}
g_{\zeta,\nu}(\omega) = \lim_{t\to +\infty}\frac{m^{\mathrm{pc}}_{G(\omega)}(u_{0,\zeta,\nu}, t\, T_\nu(Q'))}{M_\nu^{n-1}{t^{n-1}}}
=\lim_{t\to +\infty}\frac{m^{\mathrm{pc}}_{G(\omega)}(u_{0,\zeta,\nu}, tM_\nu Q^\nu(0))}{(tM_\nu )^{n-1}}
\end{equation*}
for every $\omega\in \Omega_{\zeta,\nu}$. Let $\widetilde\Omega$ be the intersection of the sets $\Omega_{\zeta,\nu}$ for $\zeta\in \Q^m_0$ and $\nu\in \mathbb{Q}^n\cap \Sph^{n-1}$. Clearly $\widetilde \Omega\in\T$ and $P(\widetilde \Omega)=1$.

We now consider the auxiliary functions $\underline{g}, \overline{g}\colon \widetilde\Om\times \R^m_0\times \Sph^{n-1}\to [ 0,+\infty]$ 
defined as
\begin{align}
\label{c:liminf}
&\underline{g}(\om,\zeta,\nu):=\liminf_{t\to +\infty}\frac{m^{\mathrm{pc}}_{G(\om)}(u_{0,\zeta,\nu},Q_{t}^\nu(0))}{t^{n-1}}
\\
\label{c:limsup}
&\overline{g}(\om,\zeta,\nu):=\limsup_{t\to +\infty}\frac{m^{\mathrm{pc}}_{G(\om)}(u_{0,\zeta,\nu},Q_{t}^\nu(0))}{t^{n-1}}
\end{align}
and note that $\underline{g}(\om,\zeta,\nu)=\overline{g}(\om,\zeta,\nu)=g_{\zeta,\nu}(\omega)$ for every $\om\in\widetilde\Omega$, $\zeta\in \Q^m_0$, and $\nu\in \mathbb{Q}^n\cap \Sph^{n-1}$.

By Lemma \ref{continuity-g-bar}  for every $\om\in \widetilde\Om$ and every $\nu\in \Sph^{n-1}$  the functions 
$\zeta\mapsto \underline{g}(\om,\zeta,\nu)$ and $\zeta\mapsto \overline{g}(\om,\zeta,\nu)$ are continuous on $\R^m_0$,  and their modulus of continuity does not depend on $\om$ and~$\nu$. This implies that 
\begin{equation} \label{equality on R}
\underline{g}(\om,\zeta,\nu)=\overline{g}(\om,\zeta,\nu)
\hbox{ for every }\om\in\widetilde\Omega,\ \zeta\in \R^m_0,\hbox{ and }\nu\in \mathbb{Q}^n\cap \Sph^{n-1},
\end{equation}
and that the function $\omega\mapsto \overline{g}(\om,\zeta,\nu)$ is $\T$-measurable on $\widetilde\Omega$ for every $\zeta\in \R^m_0$ and $\nu\in \mathbb{Q}^n\cap \Sph^{n-1}$.

Let $\Sph^{n-1}_\pm$ and $\widehat\Sph^{n-1}_\pm$ be the sets defined in (b), Section 2. It is known that $\mathbb{Q}^n\cap \Sph^{n-1}$ is dense in $\Sph^{n-1}$(see, e.g., \cite[Remark A.2]{CDMSZ}). Since $\Sph^{n-1}_\pm$ is open in the relative topology of $\Sph^{n-1}$ and is dense in  $\widehat\Sph^{n-1}_\pm$, we conclude that $\mathbb{Q}^n\cap \Sph^{n-1}_\pm$ is dense in $\widehat\Sph^{n-1}_\pm$.

Since, for fixed $\om\in \widetilde\Om$, the function $\underline g$ in \eqref{c:liminf} coincides with $\undertilde g$ in \eqref{C:g-utilde} (for $G=G(\om)$) evaluated at $x=0$, while $\overline g$ in \eqref{c:limsup} coincides with $\widetilde g$ in \eqref{C:g-tilde} (for $G=G(\om)$) evaluated at $x=0$, by Lemma~\ref{continuity_in_zeta-nu}, for every $\om\in \widetilde\Om$ and $\zeta\in \R^m_0$ the restrictions of the functions $\nu\mapsto \underline{g}(\om,\zeta,\nu)$ and $\nu\mapsto \overline{g}(\om,\zeta,\nu)$  to the sets  $\widehat\Sph^{n-1}_+$ and $\widehat\Sph^{n-1}_-$ are continuous. Therefore \eqref{equality on R} and the density of  $\mathbb{Q}^n\cap \Sph^{n-1}_\pm$ in $\widehat\Sph^{n-1}_\pm$ imply that
\begin{equation} \label{equality on sphere}
\underline{g}(\om,\zeta,\nu)=\overline{g}(\om,\zeta,\nu)
\hbox{ for every }\om\in\widetilde\Omega,\ \zeta\in \R_0^m,\hbox{ and }\nu\in \Sph^{n-1},
\end{equation}
and that the function  $\omega\mapsto \overline{g}(\om,\zeta,\nu)$ is $\T$-measurable on $\widetilde\Omega$ for every $\zeta\in \R^m_0$ and $\nu\in \Sph^{n-1}$.

For every $\om\in\Om$, $\zeta\in \R^m_0$, and $\nu\in  \Sph^{n-1}$ we define
\begin{equation} \label{defghom}
g_{\mathrm{hom}}(\omega,\zeta,\nu)=
\begin{cases}
\overline{g}(\om,\zeta,\nu)&\hbox{if }\om\in\widetilde\Omega,
\\
c_4&\hbox{if }\om\in\Om\setminus\widetilde\Omega.
\end{cases}
\end{equation}
By \eqref{equality on sphere} we may deduce \eqref{lim:density} for every $\om\in\widetilde\Omega$, $\zeta\in \R^m_0$, and $\nu\in \Sph^{n-1}$. Moreover, we have proved that 
\begin{align*}
\om&\mapsto \bar g(\om,\zeta,\nu) \,\, \textrm{is } \T\textrm{-measurable in } \widetilde \Om \,\, \textrm{ for every }\, \zeta\in \R^m_0 \,\textrm{ and } \nu \in \Sph^{n-1}, \\
(\zeta,\nu)  &\mapsto \bar g(\om,\zeta,\nu) \,\, \textrm{is continuous in } \R^m_0\times  \widehat\Sph^{n-1}_\pm  \quad \textrm{ for every }\,  \om \in \widetilde \Om.
\end{align*}
Therefore the $\T$-measurability of the function $\omega\mapsto \overline{g}(\om,\zeta,\nu)$ in $\widetilde\Om$ for every $\zeta\in \R^m_0$ and $\nu\in \Sph^{n-1}$ implies that  the restriction of $\overline{g}$ to $\widetilde\Omega\times\R^m_0\times \widehat\Sph^{n-1}_\pm$ is  
measurable with respect to the $\sigma$-algebra induced in $\widetilde\Omega\times\R^m_0\times \widehat\Sph^{n-1}_\pm$ by $\T \otimes \B^m\otimes \B^{n}_S$. 
This implies the $(\T\otimes \B^m\otimes \B^{n}_S)$-measurability of $g_{\mathrm{hom}}$ on $\Omega\times\R^m_0\times \Sph^{n-1}$, thus showing that 
$g_{\mathrm{hom}}$ satisfies property (c) of Definition \ref{ri}.

Note now that for every $\om\in\Om$ the function 
$(x,\zeta,\nu)\mapsto g_{\mathrm{hom}}(\om,\zeta,\nu)$ defined in \eqref{defghom} belongs to the class $\G$. Indeed, for $\om\in\widetilde\Omega$ this follows from Lemma \ref{continuity-g-bar} while for 
$\om\in\Om\setminus\widetilde\Omega$ this follows from the definition of $g_{\mathrm{hom}}$.
Thus, $g_{\mathrm{hom}}$ satisfies property (d) of Definition \ref{ri}, and this concludes the proof.
\end{proof}

\section{Proof  of the formula for the surface integrand: the  general case}\label{not special}

In this section we extend Theorem \ref{Ergodic-thm-surf} to the case of arbitrary $x\in\R^n$, 
thus concluding the proof of \eqref{phi0}. More precisely, we prove the following result.
\begin{thm}\label{Ergodic-thm-surf2}
Let $g$ be a stationary random surface integrand
with respect to a group $(\tau_z)_{z\in \Z^n}$ of $P$-preserving transformations on $(\Om,\T,P)$. 
Then there exist $\Om'\in\T$, with $P(\Om')=1$, and a random surface  integrand   $g_{\mathrm{hom}}\colon \Om\times \R_0^m \times \Sph^{n-1} \to \R$, independent of $x$, such that
\begin{equation}\label{Key:g}
g_{\mathrm{hom}}(\omega,\zeta,\nu)=
\lim_{t\to +\infty} \frac{m^{\mathrm{pc}}_{G(\om)}(u_{tx,\zeta,\nu},Q^\nu_{ r(t)}(tx))}{ r(t)^{n-1}},
\end{equation} 
for every $\om\in \Om'$, $x\in\R^n$, $\zeta\in \R^m_0$, $\nu\in \Sph^{n-1}$, and for every function $r\colon(0,+\infty)\to(0,+\infty)$ with $r(t)\ge t$ for $t>0$.
Moreover, if $(\tau_z)_{z\in \Z^n}$ is ergodic, then $g_{\mathrm{hom}}$ does not depend on $\om$ and  
\begin{equation}\label{ergodic6}
g_{\mathrm{hom}}(\zeta,\nu)= \lim_{t\to +\infty} \frac{1}{ r(t)^{n-1}} 
\int_{\Omega}m^{\mathrm{pc}}_{G(\om)}(u_{0,\zeta,\nu},Q^\nu_{ r(t)}(0))\,dP(\omega).
\end{equation} 
\end{thm}  
The first step in the proof of the above statement is the following invariance result.
In the ergodic case this implies that the function $g_{\mathrm{hom}}$
does not depend on $\om$ (see Corollary \ref{ergodic5}).
\begin{thm}\label{invariance3}
Let $g$ be a stationary random surface integrand
with respect to a group $(\tau_z)_{z\in \Z^n}$ of $P$-preserving transformations on $(\Om,\T,P)$, and
let $\widetilde\Om$ and $g_{\mathrm{hom}}$ be as in Theorem \ref{Ergodic-thm-surf}. Then there exists a set $\widehat{\Omega}\in\T$, with $P(\widehat{\Omega})=1$, $\widehat \Om \subset \widetilde\Om$,  and $\tau_z(\widehat{\Omega})=\widehat{\Omega}$ for every $z\in\Z^n$, such that
\begin{equation}\label{c:invariance}
g_{\mathrm{hom}}(\tau_z(\om),\zeta,\nu)=g_{\mathrm{hom}}(\om,\zeta,\nu)
\end{equation}
for every $z\in \Z^n$, $\omega\in \widehat{\Omega}$, $\zeta\in \R^m_0$, and $\nu\in \Sph^{n-1}$.
\end{thm}

\begin{proof} 
Let us define 
$$
\widehat{\Omega}:= \bigcap_{z\in \mathbb{Z}^n} \tau_z (\widetilde{\Omega}).
$$
Then $\tau_z(\widehat{\Omega})=\widehat{\Omega}\subset \widetilde{\Omega}$ for every $z\in\Z^n$, and $P(\widehat{\Omega})=1$ by property (c) of Definition \ref{P-preserving}. 
To prove \eqref{c:invariance} it is enough to show that 
\begin{equation}\label{c:invariance-enough}
g_{\mathrm{hom}}(\tau_z(\om),\zeta,\nu)\leq g_{\mathrm{hom}}(\om,\zeta,\nu)
\end{equation}
for every $z\in \Z^n$, $\omega\in \widehat{\Omega}$, $\zeta\in \R^m_0$, and $\nu\in \Sph^{n-1}$.
Indeed, the opposite inequality is obtained by applying  \eqref{c:invariance-enough} with $\om$ replaced by $\tau_z(\om)$ and $z$ replaced by $-z$.

Let  $z\in \mathbb{Z}^n$, $\om\in\widehat{\Omega}$,  $\zeta\in \R^m_0$,   and  $\nu\in \Sph^{n-1}$ be fixed. 
 For every $t> 3|z|$, let $u_t \in SBV_{\mathrm{pc}}(Q^\nu_{t}(0),\R^m)$ be such that $u_t =u_{0,\zeta,\nu}$ in a neighbourhood of $\partial Q^\nu_t(0)$, and
\begin{equation}\label{quasi min t}
G(\om)(u_t ,Q^\nu_{t}(0)) \leq m^{\mathrm{pc}}_{G(\om)}(u_{0,\zeta,\nu},Q^\nu_t(0))+ 1.
\end{equation}
 Since $g$ is stationary,  using  \eqref{covariance}  and  a change of variables  we obtain

\begin{equation}\label{trasl:1}
m^{\mathrm{pc}}_{G(\tau_z(\omega))}(u_{0,\zeta,\nu},Q_{t}^\nu(0))=m^{\mathrm{pc}}_{G(\omega)}(u_{z,\zeta,\nu},Q_{t}^\nu(z)).
\end{equation}

We now modify $u_t $ to obtain a competitor  for a   minimisation problem  related to  the right-hand side of \eqref{trasl:1}. 
Noting that $Q^\nu_t(0)  \subset\subset  Q^\nu_{t+3 |z|}(z)$ we define
$$
v_t (y):= 
\begin{cases}
u_t (y) &\textrm{if }\; y\in Q^\nu_t(0),
\cr
u_{z,\zeta,\nu}(y) &\textrm{if }\; y\in Q^{\nu}_{t+3 |z|}(z)\setminus Q^\nu_{t}(0).
\end{cases} 
$$
Clearly $v_t \in SBV_{\mathrm{pc}}(Q_{t+3 |z|}^\nu(z),\R^m)$ and $v_t =u_{z,\zeta,\nu}$ in a neighbourhood of $\partial Q^\nu_{t+3 |z|}(z)$.
It is easy to see that $S_{v_t }= S_{u_t }\cup \Sigma_1\cup \Sigma_2$, where 
$$
\Sigma_1:=\big\{y\in \partial Q^\nu_{t}(0): \big(y{\,\cdot\,}\nu\big)\big((y-z){\,\cdot\,}\nu\big)<0\big\}\quad \text{and}  \quad \Sigma_2:=\Pi^\nu_{z}\cap(Q^{\nu}_{t+3 |z|}(z)\setminus Q^\nu_{t}(0)).
$$%
Moreover $|[v_t ]|=|\zeta|$ $\hs^{n-1}$-a.e.\ on $\Sigma_1\cup \Sigma_2$. Since $3 |z|<t$, we
have $\hs^{n-1}(\Sigma_1)=2 (n-1) |z\cdot \nu| \,t^{n-2}$ and $\hs^{n-1}(\Sigma_2)=(t+3 |z|)^{n-1}-t^{n-1}\le
 3(n-1)|z|(t+3 |z|)^{n-2}< 2^{ n }(n-1) |z|\,t^{n-2}$.
Therefore $(g6)$ gives
$$
G(\om)(v_t ,Q^{\nu}_{t+3 |z|}(z)) \le G(\om)(u_t ,Q^{\nu}_{t}(0))+ M_{\zeta,z}\,t^{n-2},
$$%
where $M_{\zeta,z}:=c_5 (n-1)  (2+2^{n})  |z|(1+|\zeta|)$.
This inequality, combined with
\eqref{quasi min t} and with the definition of  $m^{\mathrm{pc}}_{G(\om)}$, gives
\begin{equation}\label{c:nun}
m^{\mathrm{pc}}_{G(\omega)}(u_{z,\zeta,\nu},Q_{t+3 |z|}^\nu(z))\leq m^{\mathrm{pc}}_{G(\om)}(u_{0,\zeta,\nu},Q^\nu_t(0))+  1  +  M_{\zeta,z}\,t^{n-2}.
\end{equation}
 Recalling that  $\tau_z(\omega)\in\widehat{\Omega}\subset \widetilde\Omega$, by \eqref{lim:density} and \eqref{trasl:1} we get
\begin{align*}
g_{\mathrm{hom}}(\tau_z(\omega),\zeta,\nu) &= \lim_{t\to +\infty} \frac{m^{\mathrm{pc}}_{G(\tau_z(\omega))}(u_{0,\zeta,\nu},Q_{t}^\nu(0))}{t^{n-1}}
= \lim_{t\to +\infty} \frac{m^{\mathrm{pc}}_{G(\omega)}(u_{z,\zeta,\nu},Q_{t}^\nu(z))}{t^{n-1}}
\\
&= \lim_{t\to +\infty} \frac{m^{\mathrm{pc}}_{G(\omega)}(u_{z,\zeta,\nu},Q^\nu_{t+3 |z|}(z))}{t^{n-1}},
\end{align*}
where in the last equality we have used the fact that $t^{n-1}/(t+3 |z|)^{n-1}  \to  1$ as $t\to+\infty$. 
 Therefore, dividing all terms of    \eqref{c:nun}   by $t^n$  and taking the limit as $t\to+\infty$ we obtain from  \eqref{lim:density}  the inequality 
$$
g_{\mathrm{hom}}(\tau_z(\omega),\zeta,\nu)\leq g_{\mathrm{hom}}(\omega,\zeta,\nu),
$$
 which proves  \eqref{c:invariance-enough}.
\end{proof}
 The  next result shows that, in the ergodic case, 
the function $g_{\mathrm{hom}}$ is independent of $\omega$.

\begin{cor}\label{ergodic5} In addition to the assumptions of Theorem \ref{invariance3}, suppose that the group 
$(\tau_z)_{z\in \Z^n}$ of $P$-preserving transformations on $(\Om,\T,P)$ is ergodic. Then there exist a set $\widehat{\Om}_0\in\T$ with $\widehat{\Om}_0\subset \widehat{\Om}$ and $P(\widehat{\Om}_0)=1$, and a surface integrand $\hat g_{\mathrm{hom}}\in\G$, independent of $x$, such that   $g_{\mathrm{hom}}(\om,\zeta,\nu)=\hat g_{\mathrm{hom}}(\zeta,\nu)$ for every $\om\in \widehat{\Om}_0$, $\zeta\in\R^m_0$, and $\nu\in  \Sph^{n-1}$.
\end{cor}

\begin{proof} 
Let $\widehat \Omega$ be as in Theorem \ref{invariance3}. We start by showing that for every $\zeta\in\Q^m_0$ and $\nu\in \Q^n\cap \Sph^{n-1}$
there exist $\hat g_{\mathrm{hom}}(\zeta,\nu) \in \mathbb{R}$ and 
a set $\widehat{\Om}^{\zeta,\nu}\in\T$,  with $\widehat{\Om}^{\zeta,\nu}\subset \widehat{\Om}$ 
and $P(\widehat{\Om}^{\zeta,\nu})=1$,  such that
\begin{equation*} %\label{claim01}
g_{\mathrm{hom}}(\om,\zeta,\nu)=\hat g_{\mathrm{hom}}(\zeta,\nu)
\quad \text{ for every }  \om\in \widehat{\Om}^{\zeta,\nu}.
\end{equation*}
To this end we fix  $\zeta\in\Q^m_0$ and $\nu\in \Q^n\cap \Sph^{n-1}$ and 
for every  $c\in \R$ we define
$$
E^{\zeta,\nu}_c:=\{ \omega \in \widehat{\Om} :  g_{\mathrm{hom}}(\om,\zeta,\nu) \geq c\}.
$$%
From \eqref{c:invariance} it follows that $\tau_z( E_c^{\zeta,\nu}) = E_c^{\zeta,\nu}$ 
for every $z\in \Z^n$. 
Since $(\tau_z)_{z\in \Z^n}$ is ergodic, we can only have 
\begin{equation}\label{zero-one}
P(E_c^{\zeta,\nu})=0\quad\text{or}\quad P(E_c^{\zeta,\nu})=1.
\end{equation}
 Since $E_{c_1}^{\zeta,\nu}\supset E_{c_2}^{\zeta,\nu}$ for $c_1<c_2$, by \eqref{zero-one}
there exists  $c_0(\zeta, \nu) \in \mathbb{R}$  such that $P(E_c^{\zeta,\nu})=0$  for $c>c_0(\zeta, \nu)$ and $P(E_c^{\zeta,\nu})=1$ for $c<c_0(\zeta, \nu)$. It follows that there exists $ \widehat{\Om}^{\zeta,\nu} \subset \widehat{\Om}$, 
with $P(\widehat{\Om}^{\zeta,\nu})=1$, such that
\begin{equation} \label{P-a.e.}
 g_{\mathrm{hom}}(\om,\zeta,\nu) =c_0 (\zeta, \nu) \quad \text{ for  every } \omega \in  \widehat{\Om}^{\zeta,\nu}.
\end{equation}
We  define  $\widehat{\Omega}_0$  as the intersection of all sets  $\widehat{\Omega}^{\zeta,\nu}$  for
 $\zeta\in \Q^m_0$  and 
$\nu\in \mathbb{Q}^n\cap \Sph^{n-1}$. Then   $\widehat{\Om}_0\subset \widehat{\Om}$ and  $P(\widehat{\Om}_0)=1$.
 We now fix $\om_0\in \widehat{\Om}_0$ and define 
$\hat g_{\mathrm{hom}}(\zeta,\nu):=g_{\mathrm{hom}}(\om_0,\zeta,\nu)$ for every $\zeta\in\R^m_0$ and every $\nu\in \Q^n\cap \Sph^{n-1}$. By \eqref{P-a.e.} we have 
$$
g_{\mathrm{hom}}(\om,\zeta,\nu)= \hat g_{\mathrm{hom}}(\zeta,\nu) \quad
 \text{ for every } \om \in  \widehat{\Omega}_0 ,\  \zeta\in\Q^m_0 ,\  \nu\in \Q^n\cap \Sph^{n-1}.
$$%
The conclusion now follows from the continuity of $(\zeta,\nu)\mapsto g_{\mathrm{hom}}(\om,\zeta,\nu)$ on $\R^m_0\times \widehat\Sph^{n-1}_\pm$ obtained in the proof of Theorem~\ref{Ergodic-thm-surf}.
\end{proof}
We now state some classical results from Probability Theory, which will be crucial for the proof of 
Theorem~\ref{Ergodic-thm-surf2}.
For every $\psi \in L^1(\Omega,\T,P)$ and for every $\sigma$-algebra $\T'\subset\T$, 
we will denote by $\mathbb{E}[\psi|\T']$ the conditional expectation of $\psi$
with respect to $\T'$. This is the unique random variable  in $L^1(\Omega,\T',P)$ 
with the property that
$$
\int_{E} \mathbb{E}[\psi| \T'](\om) \, d P (\om)
= \int_{E} \psi (\om) \, d P (\om) 
\quad \text{ for every } E \in \T'.
$$
We start by stating Birkhoff's Ergodic Theorem (for a proof, see,  e.g.,  \cite[Theorem 2.1.5]{K}). 
\begin{thm}[Birkhoff's Ergodic Theorem]\label{birk} 
Let $(\Omega,\T,P)$ be a probability space, let $T\colon\Omega\to \Omega$ be a $P$-preserving transformation, and let $\mathscr{I}_P(T)$ be the $\sigma$-algebra of $T$-invariants sets. Then for every $\psi \in L^1(\Omega,\T,P)$ we have 
$$
\lim_{k\to + \infty} \frac1k \sum_{i=1}^{k} \psi(T^i ( \omega) ) = \mathbb{E}[\psi| \mathscr{I}_P(T)](\om)
$$
for $P$-a.e. $\omega \in \Omega$.
\end{thm}
We also recall the Conditional Dominated Convergence Theorem, 
whose proof can be found in \cite[Theorem 2.7]{Bhatta}.
\begin{thm}[Conditional Dominated Convergence]\label{CDC}
Let $ \T' \subset \T$ be a $\sigma$-algebra and let $(\varphi_k)$ be a sequence of  random variables in $(\Omega,\T,P)$ converging  pointwise $P$-a.e.\ in $\Om$ to a random variable $\varphi$. Suppose that there exists $\psi \in L^1(\Omega,\T,P)$ such that $|\varphi_k|\leq \psi$ $P$-a.e.\ in $\Om$ for every $k$. Then 
$\mathbb{E}[\varphi_k| \T'](\om)\to \mathbb{E}[\varphi| \T'](\om)$ for $P$-a.e.\ $\om\in\Om$.
\end{thm}

We are now ready to prove the main result of this section.

\begin{proof}[Proof of Theorem \ref{Ergodic-thm-surf2}] Let $g_{\mathrm{hom}}$  and $\widetilde\Om$  be as in Theorem \ref{Ergodic-thm-surf} and let $\widehat\Om$ be as in Theorem \ref{invariance3}. We will prove the existence of a set $\Omega'\in\T$, with $\Omega'\subset \widehat\Om$ and $P(\Om')=1$,  such that \eqref{Key:g} holds for every $\om \in \Omega'$.
In the following, for every $z\in\Z^n$ the 
sub-$\sigma$-algebra of invariant sets for the measure-preserving map $\tau_{z}$
 is denoted by $\I_{z} \subset \T$
\ie
$\I_{z}:=\{E \in \T : \tau_{z}(E)=E\}.$
Also, for given $\zeta\in \R^m_0$, $\nu\in \Sph^{n-1}$, 
  $\eta>0$, 
we define the sequence of events $(E^{\zeta,\nu,\eta}_j)_{j\in\N}$ as
$$
E^{\zeta,\nu,\eta}_j:=\Big\{\om\in \Om : \Big|\frac{m^{\mathrm{pc}}_{G(\om)}(u_{0,\zeta,\nu},Q^{\nu}_k(0))}{k^{n-1}}-g_{\mathrm{hom}}(\omega,\zeta,\nu)\Big|\leq \eta \hbox{ for every integer } k \geq j \Big\}.
$$
We divide the proof into several steps.  We use the notation for the integer part introduced in (q), Section~\ref{Notation}.
\medskip

\noindent
\textbf{Step 1.}  Let us fix  $z\in\Z^n$, $\zeta\in \R^m_0$, $\nu\in \Sph^{n-1}$, 
 and $\eta>0$. We prove that  there exists a set $ \widetilde\Om^{\zeta,\nu,\eta}_{z}  \in \T$, with $\widetilde\Om^{\zeta,\nu,\eta}_{z}\subset\widehat{\Om}$ and $P( \widetilde\Om^{\zeta,\nu,\eta}_{z})=1$, satisfying the following property:

for every $\delta>0$ and every   $\om\in\widetilde\Om^{\zeta,\nu,\eta}_{z}$
there exists an integer $j_0=j_0(\zeta,\nu,\eta,z,\om, \delta)$ 
such that
\begin{equation}\label{enne_0}
\mathbb{E}[\chi_{E^{\zeta,\nu,\eta}_{j_0}} |\I_{z}](\om) >1-\delta. 
\end{equation}

To prove \eqref{enne_0}  we apply  Theorem \ref{Ergodic-thm-surf}  and we obtain  (since $\widehat{\Omega}\subset \widetilde{\Omega}$)
\begin{equation*}%\label{conv-caratteristica}
\lim_{j\to+\infty}\chi_{E^{\zeta,\nu,\eta}_{j}}(\om) = 1 \quad\text{for every }\om\in \widehat\Om.
\end{equation*}
 By the  Conditional Dominated Convergence Theorem \ref{CDC}   there exists  a  set $\widetilde\Om^{\zeta,\nu,\eta}_{z} \in \T$, with $ \widetilde\Om^{\zeta,\nu,\eta}_{z}\subset \widetilde\Om$ and $P(\widetilde\Om^{\zeta,\nu,\eta}_{z})=1$, such that
\begin{equation}\label{lim1}
\lim_{j\to +\infty}\mathbb{E}[\chi_{E^{\zeta,\nu,\eta}_{j}} |\I_{z}](\om)=\mathbb{E}[1 |\I_{z}](\om)=1
\quad \text{for every }\om\in\widetilde\Om^{\zeta,\nu,\eta}_{z}.
\end{equation}
 Given  $\om\in \widetilde\Om^{\zeta,\nu,\eta}_{z}$ and $\delta>0$,  the existence of $j_0$ satisfying    \eqref{enne_0} follows  from  \eqref{lim1}. 

\medskip

\noindent
\textbf{Step 2.}  Let $z$, $\zeta$, $\nu$, and
$\eta$ be as in Step 1 and let $0<\delta<\frac14$.
We prove that there exist a set $\Om^{\zeta,\nu,\eta}_{z} \in \T$, with $\Om^{\zeta,\nu,\eta}_{z}\subset\widetilde\Om^{\zeta,\nu,\eta}_{z}$ and $P(\Om^{\zeta,\nu,\eta}_{z})=1$, and an integer constant $m_0=m_0(\zeta,\nu,\eta,z,\om, \delta)> \frac1\delta$ satisfying the following property:

for every $\om\in \Om^{\zeta,\nu,\eta}_{z}$ and for every integer $m\ge m_0$ there exists $ i =i(\zeta,\nu,\eta,z,\om, \delta, m) \in \{m+1,\dots,m+\ell\}$, with $\ell:= \lfloor 5m\delta\rfloor$, such that 
\begin{equation}\label{stima-centro-sbagliato}
\bigg|\frac{m^{\mathrm{pc}}_{G(\om)}(u_{-i z,\zeta,\nu},Q^{\nu}_k(-i z))}{k^{n-1}}-g_{\mathrm{hom}}(\om,\zeta,\nu)\bigg|\leq \eta \qquad \text{for every\ }  k\geq j_0,
\end{equation}
 where $j_0=j_0(\zeta,\nu,\eta,z,\om, \delta)$ is the  integer introduced in Step 1.

To prove \eqref{stima-centro-sbagliato} we apply  Birkhoff Ergodic Theorem \ref{birk}  with $\psi:=\chi_{E^{\zeta,\nu,\eta}_{j}}$ and $T:=\tau_{z}$, and we obtain that
there exists a set $\Om^{\zeta,\nu,\eta}_{z} \in \T$, with $\Om^{\zeta,\nu,\eta}_{z}\subset\widetilde\Om^{\zeta,\nu,\eta}_{z}$ and $P(\Om^{\zeta,\nu,\eta}_{z})=1$, such that

\begin{equation}\label{lim2}
\lim_{m\to+\infty}\frac{1}{m}\sum_{i=1}^m \chi_{E^{\zeta,\nu,\eta}_{j}}(\tau_{iz}(\om))=\mathbb{E}[\chi_{E^{\zeta,\nu,\eta}_{j}} |\I_{z}](\om)
\end{equation}
for every $j \in \mathbb{N}$ and every $\om\in\Om^{\zeta,\nu,\eta}_{z}$. 
In particular, for a given $\om\in\Om^{\zeta,\nu,\eta}_{z}$, equality \eqref{lim2} holds for the index $j_0=j_0(\zeta,\nu,\eta,z,\om, \delta)$
introduced in Step 1. 
 Therefore,  there exists an integer $\hat m=\hat m(\zeta,\nu,\eta,z,\om, \delta)$ such that
\begin{equation}\label{BET}
\frac{1}{m}\sum_{i=1}^m \chi_{E^{\zeta,\nu,\eta}_{j_0}}(\tau_{iz}(\om))
>
\mathbb{E}[\chi_{E^{\zeta,\nu,\eta}_{j_0}} |\I_{z}](\om)- \delta
\qquad \text{for every } \, m\geq \hat m.
\end{equation}
Fix now an integer $m\ge m_0:= \max\{2\hat m, 2j_0,\lfloor\frac1\delta\rfloor+1\}$ and set $\ell:= \lfloor 5m\delta\rfloor$. 
We claim that 
\begin{equation} \label{existence i}
\text{there exists } i =i(\zeta,\nu,\eta,z,\om, \delta, m) \in \{m+1,\dots,m+\ell\}
\text{ such that } \tau_{i z}(\om) \in E^{\zeta,\nu,\eta}_{j_0}.
\end{equation}
Suppose, by contradiction, that \eqref{existence i} fails.
Then, we have 
$$
\tilde \ell := \# \{ i \in \mathbb{N}, 1 \leq i \leq m:  \chi_{E^{\zeta,\nu,\eta}_{j_0} }(\tau_{iz}(\om))=1\}
= \# \{ i \in \mathbb{N}, 1 \leq i \leq m + \ell:  \chi_{E^{\zeta,\nu,\eta}_{j_0} }(\tau_{iz}(\om))=1\}.
$$
So, \eqref{BET} with $m$ replaced by $m+\ell$ gives
\begin{equation}\label{m+l}
\frac{\tilde\ell}{m+\ell} 
= \frac{1}{m+\ell}\sum_{i=1}^{m+\ell} \chi_{E^{\zeta,\nu,\eta}_{j_0}}(\tau_{iz}(\om))
> \mathbb{E}[\chi_{E^{\zeta,\nu,\eta}_{j_0}} |\I_{z}](\om)- \delta.
\end{equation}
Therefore, using \eqref{enne_0} and \eqref{m+l} we obtain
\begin{equation}\label{bound-on-L}
\delta > \mathbb{E}[\chi_{E^{\zeta,\nu,\eta}_{j_0}} |\I_{z}](\om) - \frac{\tilde\ell}{m+\ell}
=\mathbb{E}[\chi_{E^{\zeta,\nu,\eta}_{j_0}} |\I_{z}](\om)-1+\frac{\ell+m-\tilde\ell}{m+\ell}> \frac{\ell+m-\tilde\ell}{m+\ell} -\delta.
\end{equation}
Since $m-\tilde\ell \geq 0$, from \eqref{bound-on-L} we deduce that 
$\ell(1-2\delta)< 2m\delta$. This, using the fact that $\delta < \frac14$,  
gives $\ell<4m\delta$.
On the other hand, by definition $\ell= \lfloor 5m\delta\rfloor \geq 5m\delta -1 >  4 m\delta$, since $m>\frac1\delta$. This  contradicts the inequality  $\ell < 4m\delta$ and  proves \eqref{existence i}. 
As a consequence, by the definition of $ E^{\zeta,\nu,\eta}_{j_0}$,
$$
\bigg|\frac{m^{\mathrm{pc}}_{G(\tau_{i z}(\om))}(u_{0,\zeta,\nu},Q^{\nu}_k(0))}{k^{n-1}}-g_{\mathrm{hom}}(\tau_{i z}(\om),\zeta,\nu)\bigg|\leq \eta
$$
for every integer $k\geq j_0$.
Since $\om\in\Om^{\zeta,\nu,\eta}_{z}\subset \widehat{\Om}$, 
thanks to \eqref{c:invariance} and \eqref{trasl:1} we get \eqref{stima-centro-sbagliato}.

\medskip

\noindent
\textbf{Step 3.} We show that the result we want to prove is true along integers.  More precisely, we prove that 
there exists $\Om' \in \mathcal{T}$,  with $\Om' \subset \Om$ and $P(\Om')=1$,  such that
\begin{equation}\label{lim on integers}
\lim_{\substack{ k  \to+\infty\\  k  \in\N}}\frac{m^{\mathrm{pc}}_{G(\om)}
(u_{- k   z,\zeta,\nu},Q^{\nu}_{ m_k }(- k   z))}{
 m_k^{ n-1}}=g_{\mathrm{hom}}(\om,\zeta,\nu)
\end{equation}
for every $\omega \in \Om'$, $z \in\Z^n$, $\zeta\in \Q^m_0$, $\nu\in \Q^n\cap \Sph^{n-1}$, and for every 
sequence of integers $(m_k)$ such that $m_k\ge k$ for every~$k$.

 To prove this property, we define $\Om'$ as the intersection of the sets $\Om^{\zeta,\nu,\eta}_{z}$ (introduced in Step~2) for $\zeta\in \Q_0^m$, $\nu\in \Q^n\cap \Sph^{n-1}$, $\eta\in \Q$, with $\eta>0$, and $z\in\Z^n$. It  is clear that $\Omega'\subset \widehat\Om$ and $P(\Om')=1$. 
Let us fix $\omega$, $z$, $\zeta$, $\nu$ and $(m_k)$ as required. Moreover, let us fix $\delta>0$, with $20\,\delta\,(|z|+1)<1$, and $\eta\in \Q$, with $\eta>0$. Let $m_0=m_0(\zeta,\nu,\eta,z,\om, \delta)$ be as in Step 2.
For every $k\ge 2 m_0$ let $\underline m_k, \overline m_k\in \Z$   be defined as 
$$
 \underline m_k  :=  m_k -2(i_k- k )\lfloor |z|+1\rfloor \qquad \textrm{and} \qquad   \overline m_k  :=  m_k +2(i_k- k ) \lfloor|z|+1\rfloor,
$$%
 where $i_k=i(\zeta,\nu,\eta,z,\om, \delta, k)$ is the index introduced in Step 2 corresponding to $m=k$. 
Clearly $ \underline m_k  \leq m_k  \leq  \overline m_k $. Moreover, since  $|z|<\lfloor |z|+1\rfloor$, we have that % $ \underline m_k  \geq \frac m2\ge m_0$ and
\begin{equation}\label{cubi-inscatolati}
Q^\nu_{ \underline m_k } (-i_k  z) \subset\subset Q^\nu_{ m_k} (- k z) \subset\subset Q^\nu_{ \overline m_k } (-i_k  z).
\end{equation}

Let us now compare  the  minimisation problems for $G(\om)$ relative to the cubes in \eqref{cubi-inscatolati}. 
For every $k$  let $ u_k\in SBV_{\mathrm{pc}}(Q^{\nu}_{ m_k}(- k  z),\R^m)$ be such that with $u_k=u_{- k  z,\zeta,\nu}$ in a neighbourhood of $\partial Q^{\nu}_{ m_k  }(- k  z)$ and
\begin{equation}\label{quasi min}
G(\om)(u_k,Q^{\nu}_{  m_k}(- k  z))\le 
m^{\mathrm{pc}}_{G(\om)}(u_{- k  z,\zeta,\nu},Q^{\nu}_{  m_k }(- k   z))+\eta;
\end{equation}
thanks to \eqref{cubi-inscatolati} the extension of $u_k$ defined as 
$$
v_k(y):=\begin{cases} u_k(y) & \text{if }\; y\in Q^\nu_{ m_k }(- k z),\cr
u_{-i_k z,\zeta,\nu}(y) & \text{if }\; y\in Q^\nu_{ \overline m_k }(-i_k z) \setminus 
Q^\nu_{ m_k }(- k z),
\end{cases}
$$%
belongs to $SBV_{\mathrm{pc}}(Q^{\nu}_{ \overline m_k }(-i_k z),\R^m)$ and  satisfies  $ v_k=u_{-i_k z,\zeta,\nu}$ in a neighbourhood of $\partial Q^{\nu}_{ \overline m_k }(-i_k z)$.
By  the  definition of $v_k$ it follows that $S_{v_k}\subset S_{u_k}\cup \Sigma_k^1\cup \Sigma_k^2$, where 
\begin{align*}
&\Sigma_k^1:= \big\{y\in \partial Q^\nu_{ m_k }(- k z): \big((y+ k z){\,\cdot\,}\nu\big) \big((y+i_k z){\,\cdot\,}\nu\big) <0\big\},
\\
&\Sigma_k^2:=\Pi^\nu_{-i_k z}\cap(Q^\nu_{ \overline m_k }(-i_k z) \setminus Q^\nu_{ m_k }(- k z)).
\end{align*}
Moreover $|[v_k]|=|\zeta|$ $\hs^{n-1}$-a.e.\ on $\Sigma_k^1 \cup \Sigma_k^2$. Since  $ 20\delta (|z|+1)<1$, $k\le m_k$, and $i_k-k\le  5k\delta$ by \eqref{existence i},  we  obtain  $| k z-i_k z|\le  (i_k-k)|z|\le 5 k \delta |z| \le 5 m_k\delta |z|< \frac{m_k}{2}$.  Moreover,  $ \overline m_k -m_k
=  2(i_k-k) \lfloor |z|+1\rfloor \le 10  k\delta \lfloor |z|+1\rfloor  \le  10  m_k\delta \lfloor |z|+1\rfloor
 < \frac{m_k}{2}$,  hence  $ \overline m_k <  2m_{ k}$. From the previous inequalities we obtain
  $\hs^{n-1}(\Sigma_k^1)\le  10 (n-1)  \delta |z|  m^{n-1}_{ k}$ and $\hs^{n-1}(\Sigma_k^2)= \overline m_k ^{n-1}-  m^{n-1}_{ k}\le 5(n-1)2^{n-1}\delta\lfloor |z|+1\rfloor \,m^{n-1}_{ k}$.
Then by the growth condition $(g6)$ we have
$$
G(\om)(v_k, Q^\nu_{ \overline m_k }(-i_k z))\le G(\om)(u_k,Q^{\nu}_{ m_k}(- k z)) + C_{\zeta,z}\delta\, m^{n-1}_{ k},
$$%
where $C_{\zeta,z} := c_5 \, 5 (n-1) (2+ 2^{n-1})  \lfloor |z|+1\rfloor (1+|\zeta|)$. This inequality, combined with
\eqref{quasi min} and with the definition of  $m^{\mathrm{pc}}_{G(\om)}$, gives
\begin{equation}\label{stima-ell2}
m^{\mathrm{pc}}_{G(\om)}(u_{-i_k z,\zeta,\nu},Q^{\nu}_{ \overline m_k }(-i_k z)) 
\leq m^{\mathrm{pc}}_{G(\om)}(u_{- k z,\zeta,\nu},Q^{\nu}_{ m_k}(- k  z))+\eta+
C_{\zeta,z}\delta\,m^{n-1}_{ k }.
\end{equation}
Thus,  dividing  all  terms in \eqref{stima-ell2} by $ \overline m_k ^{n-1}$ and recalling that $ \overline m_k \geq m_k$,  we get
\begin{equation}\label{ell2}
\frac{m^{\mathrm{pc}}_{G(\om)}(u_{-i_k  z,\zeta,\nu},Q^{\nu}_{ \overline m_k }(-i_k  z))}{ \overline m_k^{n-1}} \leq \frac{m^{\mathrm{pc}}_{G(\om)}(u_{- k  z,\zeta,\nu},Q^{\nu}_{m_k }(- k  z))}{ m^{n-1}_{ k }}+\frac{\eta}{ m_k^{n-1}  } +C_{\zeta,z}\delta.
\end{equation}
By the definition of $ \underline m_k $ and since $m_k \geq  \underline m_k $, a similar argument yields 
\begin{equation}\label{ell1}
\frac{m^{\mathrm{pc}}_{G(\om)}(u_{- k  z,\zeta,\nu},Q^{\nu}_{m_k }(- k  z))}{  m^{n-1}_k } \leq \frac{m^{\mathrm{pc}}_{G(\om)}(u_{-i_k  z,\zeta,\nu},Q^{\nu}_{ \underline m_k }(-i_k  z))}{ \underline m^{n-1}_k} +\frac{\eta}{ m_k^{n-1}  } + 2  C_{\zeta,z}\delta.
\end{equation}

 Since  $\underline m_k\to +\infty$ as $k\to +\infty$ and  $\overline m_k \geq  \underline m_k$ for every $k$, we have 
$\overline m_k \geq \ \underline m_k \geq  j_0$ for $k$ large enough, where $j_0=j_0(\zeta,\nu,\eta,z,\om, \delta)$ is the integer introduced in Step 1. 
As  $\om\in \Om^{\zeta,\nu,\eta}_{z}$ , gathering \eqref{stima-centro-sbagliato}, \eqref{ell2}, and \eqref{ell1} gives 
$$
\bigg|\frac{m^{\mathrm{pc}}_{G(\om)}(u_{- k  z,\zeta,\nu},Q^{\nu}_{ m_k}(- k  z))}{ m_k ^{n-1}}-g_{\mathrm{hom}}(\om,\zeta,\nu)\bigg| \leq   \eta + \frac{\eta}{m_k^{n-1}} + 2  C_{\zeta,z}\delta
$$
 for $k$ large enough. 
We conclude that
\begin{equation*}%\label{limsup lambda delta}
\limsup_{\substack{ k \to+\infty\\ k \in\N}}\bigg|\frac{m^{\mathrm{pc}}_{G(\om)}(u_{- k  z,\zeta,\nu},Q^{\nu}_{ m_k }(- k  z))}{  m_k^{n-1}}-g_{\mathrm{hom}}(\om,\zeta,\nu)\bigg| \leq  \eta + 2 C_{\zeta,z}\delta.
\end{equation*}
 Since this inequality holds  for  every  $\delta>0$, with $20\delta( |z|+1)<1$, and every $\eta\in \Q$, with $\eta>0$, we obtain \eqref{lim on integers}.

\medskip

\noindent
\textbf{Step 4.} We show that \eqref{Key:g} holds when $\zeta$, and $\nu$ have rational coordinates. Namely,
given  $\om \in \Om'$ (the set introduced in Step~3), $x \in\R^n$, $\zeta\in \Q^m_0$, $\nu\in \Q^n\cap \Sph^{n-1}$, and a function $r\colon(0,+\infty)\to(0,+\infty)$, with $r(t)\ge t$ for every $t>0$, we prove that \eqref{Key:g} holds.

To this aim, we fix $\eta>0$. Then there exist $ q \in\Q^n$  such that $| q-x|<\eta$ and $h\in\N$ such that $z:=h q \in\Z^n$.

Let $(t_k)$ be a sequence of real numbers with $t_k \to +\infty$  and let $s_k:=t_k/h$. By the definition of $m^{\mathrm{pc}}_{G(\om)}$ for every $k$ there exists $ \hat u_k\in SBV_{\mathrm{pc}}(Q^{\nu}_{ r(t_k) }(t_k x),\R^m)$, with $ \hat u_k=u_{t_k x,\zeta,\nu}$ in a neighbourhood of $\partial Q^{\nu}_{ r(t_k) }(t_k  x)$, such that 
\begin{equation}\label{quasi min k}
G(\om)( \hat u_k,Q^{\nu}_{ r(t_k) }(t_k  x))\le m^{\mathrm{pc}}_{G(\om)}(u_{t_k x,\zeta,\nu}, Q^{\nu}_{ r(t_k) }(t_k x))+\eta.
\end{equation}
 We fix an integer $j>2|z|+1$ and define $r_k:=\lfloor r(t_k)+2\eta t_k\rfloor+j$. It is easy to check that  
\begin{equation*}%\label{inclusion cubes}
Q^\nu_{ r(t_k) }(t_k  x) \subset\subset  Q^\nu_{ r_k}( \lfloor s_k \rfloor  z).
\end{equation*}
As usual, we can extend $ \hat u_k$ to $Q^\nu_{ r_k}( \lfloor s_k \rfloor z )$ as
$$
 \hat v_k(y):=
\begin{cases} 
\hat u_k(y) & \text{if } y\in  Q^\nu_{ r(t_k) }(t_k x) \\
u_{ \lfloor s_k \rfloor  z,\zeta,\nu}(y) & \text{if } y\in Q^\nu_{ r_k}( \lfloor s_k \rfloor  z)
\setminus Q^\nu_{ r(t_k) }(t_k x).
\end{cases}
$$%
Then $ \hat v_k\in SBV_{\mathrm{pc}}(Q^\nu_{ r_k}( \lfloor s_k \rfloor  z),\R^m)$ and $ \hat v_k=u_{ \lfloor s_k \rfloor  z,\zeta,\nu}$ in a neighbourhood of 
$\partial Q^\nu_{ r_k}( \lfloor s_k \rfloor  z)$.
By  the  definition of $ \hat v_k$ it follows that $S_{ \hat v_k}= S_{ \hat u_k}\cup  \hat \Sigma^1_k\cup  \hat \Sigma^2_k$, where 
\begin{align*}
& \hat \Sigma^1_k:= \big\{y\in \partial Q^\nu_{ r(t_k) }(t_k x): \big((y  -  t_k x){\,\cdot\,}\nu\big)\big((y- \lfloor s_k \rfloor  z){\,\cdot\,}\nu\big)<0\big\}, \\ 
& \hat \Sigma^2_k:=\Pi^\nu_{ \lfloor s_k \rfloor  z}\cap(Q^\nu_{ r_k}( \lfloor s_k \rfloor  z)\setminus Q^\nu_{ r(t_k) }(t_k x)).
\end{align*}
Moreover $|[ \hat v_k]|=|\zeta|$ $\hs^{n-1}$-a.e.\ on $ \hat \Sigma^1_k\cup  \hat \Sigma^2_k$.  Since 
$|(t_kx-\lfloor s_k \rfloor z){\,\cdot\,}\nu|\le |t_k x - t_k  q| + |s_k z - \lfloor s_k \rfloor z | \le t_k\eta + |z|$ we have
  $\hs^{n-1}( \hat \Sigma^1_k)\le  2(n-1) r(t_k)^{n-2}(t_k\eta + |z|)$ and $\hs^{n-1}( \hat \Sigma^2_k)
= r_k^{n-1}- r(t_k) ^{n-1}\le (n-1)( r(t_k)+2\eta t_k+j)^{n-2}(2\eta t_k + j)$.
Then by the growth conditions $(g6)$ we have
$$
G(\om)( \hat v_k,Q^\nu_{ r_k }( \lfloor s_k \rfloor  z))\le 
G(\om)( \hat u_k,Q^\nu_{ r(t_k) }( t_k x ))+
 C_{\zeta}( r(t_k) +2\eta t_k +j)^{n-2} (2\eta t_k + j),
$$%
where $C_{\zeta}:= 2(n-1) c_5 (1+|\zeta|) $. This inequality, combined with
\eqref{quasi min k} and with the definition of  $m^{\mathrm{pc}}_{G(\om)}$, gives
\begin{equation*}%\label{stima-uk}
m^{\mathrm{pc}}_{G(\om)}(u_{ \lfloor s_k \rfloor z,\zeta,\nu},Q^\nu_{ r_k }( \lfloor s_k \rfloor  z) \leq m^{\mathrm{pc}}_{G(\om)}(u_{t_k x,\zeta,\nu}, Q^{\nu}_{ r(t_k)}(t_k x))+\eta + 
C_{\zeta} (1+3\eta)^{n-2}3\eta\, r(t_k)^{n-1},
\end{equation*}
 for $k$ large enough so that $2\eta t_k+j\le 3\eta r(t_k)$. 
Dividing  all  terms  of the previous inequality 
%in \eqref{stima-uk}
by $ r(t_k)^{n-1}$ and recalling that $r_k\geq r(t_k)$  we get
\begin{equation*}%\label{ell2}
\frac{m^{\mathrm{pc}}_{G(\om)}(u_{ \lfloor s_k \rfloor z,\zeta,\nu},
Q^\nu_{ r_k}( \lfloor s_k \rfloor  z)}{ r_k^{n-1}} \leq 
\frac{m^{\mathrm{pc}}_{G(\om)}(u_{ t_k x  ,\zeta,\nu}, 
Q^{\nu}_{ r(t_k)}(t_k x))}{ r(t_k) ^{n-1}}+\frac{\eta}{ r(t_k)^{n-1}} +
C_{\zeta} (1+3\eta)^{n-2}3\eta.
\end{equation*}
Finally, since $\omega \in \Omega'$, $r_k\in\N$, $ z \in\Z^n$, and $r_k\ge r(t_k)\ge t_k\ge s_k\ge \lfloor s_k \rfloor$, we can apply \eqref{lim on integers}:  By  taking first the limit as $k\to +\infty$ and then as $\eta\to0+$ we obtain
\begin{equation}\label{in:uno}
g_{\mathrm{hom}}(\omega,\zeta,\nu) \leq 
\liminf_{k\to +\infty} \frac{m^{\mathrm{pc}}_{G(\om)}(u_{t_k x,\zeta,\nu},  Q^{\nu}_{ r(t_k)}(t_k x))}{ r(t_k)^{n-1}}.
\end{equation}%
A similar argument   leads to 
$$
\limsup_{k\to +\infty} \frac{m^{\mathrm{pc}}_{G(\om)}(u_{t_k x,\zeta,\nu},  Q^{\nu}_{ r(t_k)}(t_k x))}{ r(t_k)^{n-1}} \leq g_{\mathrm{hom}}(\omega,\zeta,\nu),
$$
which,  combined with \eqref{in:uno},  proves that \eqref{Key:g} holds for every $\om \in \Om'$, $x \in \R^n$, $\zeta\in \Q^m_0$, and $\nu\in \Q^n\cap \Sph^{n-1}$.

\medskip

\noindent
\textbf{Step 5.}  We conclude the proof. 
We now extend this result to the general case $\zeta\in \R^m_0$ and $\nu\in \Sph^{n-1}$. To this end we fix $\om\in \Om'$ and consider the functions $\undertilde g(\om,\cdot,\cdot,\cdot)$ and $\widetilde g(\om,\cdot,\cdot,\cdot)$ defined on $\R^n\times \R^m_0\times \Sph^{n-1}$ by \eqref{C:g-utilde} and \eqref{C:g-tilde}, with $g$ replaced by  $g(\om,\cdot,\cdot,\cdot)$. In view of %what we have proven so far 
 Step  4  we  have
\begin{equation}\label{c:ug-dens-set}
\widetilde{g}(\omega,x,\zeta,\nu)=\undertilde{g}(\omega,x,\zeta,\nu)= g_{\mathrm{hom}}(\om,\zeta,\nu)
\end{equation}
for every $x\in\R^n$,  $\zeta \in \Q^m_0$, and $\nu\in \Q^n\cap \Sph^{n-1}$. 
By Lemma \ref{continuity_in_zeta-nu} and arguing as in the last part of the proof of Theorem \ref{Ergodic-thm-surf}, we obtain that \eqref{c:ug-dens-set} holds for every $x\in\R^n$,  $\zeta \in \R^m_0$, and $\nu\in \Sph^{n-1}$. This proves \eqref{Key:g} for every $\om \in \Om'$, $x \in\R^n$, $\zeta\in \R^m_0$, and $\nu\in \Sph^{n-1}$.

Moreover, if $(\tau_z)_{z\in \Z^n}$ is ergodic, then by Corollary \ref{ergodic5} the function $g_{\mathrm{hom}}$ does not depend on $\om$ and \eqref{ergodic6} can be obtained by integrating  \eqref{lim:density}  on $\Om$, and using the Dominated Convergence Theorem thanks to \eqref{stimatildemu}.
\end{proof}

\section*{Appendix. Measurability issues}%\label{measurability-surface}
\setcounter{thm}0
\renewcommand{\theequation}{A.\arabic{equation}}
\renewcommand{\thethm}{A.\arabic{thm}}

The main result of this section if the following proposition, which gives the 
measurability of the function $\om\mapsto m^{\mathrm{pc}}_{G(\om)}(w,A)$.
This property was crucial  in the proof of Proposition \ref{propr_ms}.
\begin{prop}\label{measurability} Let  $(\Om,\widehat\T,\widehat P)$ be the completion of the probability space $(\Om,\T, P)$, let
 $g$ be a stationary random surface integrand, and let $A\in\A$. Let $G(\om)$ be as in \eqref{surf}, with $g$ replaced by $g(\om,\cdot,\cdot,\cdot)$.  Let $w\in L^0(\R^n,\R^m)$ be such that $w|_A\in SBV_{\mathrm{pc}}(A,\R^m)\cap L^{\infty}(A,\R^m)$, and for every $\om\in \Om$ let $m^{\mathrm{pc}}_{G(\om)}(w,A)$ be as in \eqref{emmeG},  with $G$ replaced by $G(\om)$. Then the function $\om\mapsto m^{\mathrm{pc}}_{G(\om)}(w,A)$ is $\widehat\T$-measurable.
\end{prop}

The main difficulty in the proof of Proposition \ref{measurability} is that, although $\om\mapsto G(\om)(u,A)$ is clearly $\T$-measurable, $m^{\mathrm{pc}}_{G(\om)}(w,A)$  is defined as an infimum on an uncountable set. This difficulty is usually solved by means of the Projection Theorem, which requires the completeness of the probability space. It also requires joint measurability in $(\om,u)$ and some topological properties of the space on which the infimum is taken, like separability and metrisability. In our case (see \eqref{emmeG}) the infimum is taken on the space of all functions $u \in L^0(\R^n,\R^m)$ such that $\ u|_A\in SBV_{\mathrm{pc}}(A,\R^m)$ and $u=w$ near $\partial A$, and it is not easy to find a topology on this space with the above mentioned properties and such that  $(\om,u)\mapsto G(\om)(u,A)$ is jointly measurable. Therefore we have to attack the measurability problem in an indirect way, extending (an approximation of) $G(\om)(u,A)$ to  a  suitable subset of the space of bounded Radon measures, which turns out to be compact and metrisable in the weak$^*$ topology.

We start by introducing some notation that will be used later. For every every $A\in \A$ we denote by $\mathcal{M}_b(A,\R^{m\times n})$ the Banach space of all $ \R^{m\times n}$-valued Radon measures on $A$. This space is identified with the dual of the space $C_0(A,\R^{m\times n})$
of all $\R^{m\times n}$-valued continuous functions on $\overline A$ vanishing on $\partial A$. 
For every $R>0$ we set
$$
\mathcal M^R_A:=\{\mu \in \mathcal M_b(A,\R^{m\times n}) \colon |\mu|(A)\leq R\}, 
$$%
where $|\mu|$ denotes the variation of $\mu$ with respect to the Euclidean norm on $\R^{m\times n}$. On $\mathcal M^R_A$ we consider the topology induced by the weak$^*$ topology of $\mathcal{M}_b(A,\R^{m\times n})$.
Before starting the proof of Proposition~\ref{measurability}, we need two preliminary results.
\begin{lem}\label{l:Measurability1}
Let $(\Lambda, \mathcal S)$ be a measurable space, let $A\in\A$, let $R>0$, and let $h \colon \Lambda \times A \to \R$ be a bounded and $\mathcal S \otimes \B(A)$-measurable function. %[COMMENTO secondo (p) del capitolo 2 non va scritto l'esponente $n$ in $\B(A)$] 
Let $H\colon \Lambda \times \mathcal M^R_A \to \R$ be defined by
\begin{equation}\label{l:acca1}
H(\lambda,\mu):=\int_A h(\lambda,x)\,d|\mu|(x).
\end{equation}
Then $H$ is $\mathcal S \otimes \B(\mathcal M^R_A)$-measurable.
\end{lem}

\begin{proof}
Let $\mathcal{H}$ be the set of all bounded, $\mathcal S \otimes \B(A)$-measurable functions $h$ such that the function 
$H$ defined by \eqref{l:acca1} is $\mathcal S \otimes \B(\mathcal M^R_A)$-measurable. Clearly $\mathcal{H}$ is a monotone class (see, e.g., \cite[Definition 4.12]{AliBorder}) which contains all the functions of the form $h(\lambda,x)=\varphi(\lambda)\psi(x)$ with $\varphi$ bounded and $\mathcal S$-measurable and $\psi \in C^0_c(A)$. Then the functional form of the Monotone Class Theorem (see, e.g., \cite[Chapter I, Theorem 21]{Dellacherie}) implies that $\mathcal{H}$  coincides with the class of all bounded and $\mathcal S\otimes\B(A)$-measurable functions and this concludes the proof.
\end{proof}

\begin{cor}\label{l:Measurability2}
Let $A\in\A$, let $R>0$, and let $h \colon \Omega \times A \times \mathcal M^R_A  \to \R$ be a bounded and $\T \otimes \B(A) \otimes \B(\mathcal M^R_A)$-measurable function. Let $H\colon \Omega \times \mathcal M^R_A \to \R$ be defined by
\begin{equation*}%\label{l:acca}
H(\omega,\mu):=\int_A h(\om,x,\mu)\,d|\mu|(x).
\end{equation*}
Then $H$ is $\T \otimes \B(\mathcal M^R_A)$-measurable. 
\end{cor}

\begin{proof}
As a preliminary step, we consider the augmented functional $\tilde H\colon \Omega \times \mathcal \mathcal M^R_A \times \mathcal M^R_A \to \R$ defined by
\begin{equation*}%\label{l:acca_aug}
\tilde H(\omega,\nu,\mu):=\int_A h(\om,x,\nu)\,d|\mu|(x).
\end{equation*}
By applying Lemma \ref{l:Measurability1} to $\tilde H$, with $\Lambda = \Omega\times \mathcal M^R_A$,  $\lambda = (\om,\nu)$, and $\mathcal S= \T  \otimes \B(\mathcal M^R_A)$, we deduce that $\tilde H$ is $\T \otimes \B(\mathcal M^R_A)\otimes \B(\mathcal M^R_A)$-measurable.

The claim then follows by noting that
$
H(\omega,\mu) = \tilde H(\omega,\mu,\mu)
$
and by observing that $(\omega,\mu)\mapsto (\omega,\mu,\mu)$ is measurable for the $\sigma$-algebras $\T \otimes \B( \mathcal M^R_A )$ and $\T \otimes \B( \mathcal M^R_A)\otimes \B(\mathcal M^R_A)$.
\end{proof}

We are now ready to give the proof of Proposition \ref{measurability}.

\begin{proof}[Proof of Proposition \ref{measurability}]
For every $k\in\N$ let $m^k_{G(\om)}(w,A)$ be as in \eqref{emme-kappa}, with $G$ replaced by $G(\om)$. 
In view of \eqref{approx-min-pb-G}, the function $\om\mapsto m_{G(\om)}(w,A)$ is $\widehat\T$-measurable
if  
\begin{equation}\label{mk measurable}
\om\mapsto m^k_{G(\om)}(w,A)\quad\text{is } \widehat\T \text{-measurable}
\end{equation}
 for $k$ sufficiently large. To prove this property
 we fix $k>\|w\|_{L^\infty(A,\R^m)}$ and   observe that 
there is a one-to-one correspondence between the space of rank one $m{\times}n$ matrices and the quotient of   $\R^m_0\times \mathbb{S}^{n-1}$ with respect to the equivalence relation $(\zeta,\nu)\sim(-\zeta,-\nu)$.
Therefore, thanks to $(g6)$ and $(g7)$, for every $k\in\N$ we can define a bounded $\T\otimes\B(A)\otimes\B^{m{\times}n}$-measurable function $\tilde g_k\colon \Om\times A\times \R^{m{\times}n}\to \R$ such that 
\begin{equation*} %\label{gtilde}
\tilde g_k(\om,x,\zeta \otimes \nu)=g(\om,x,\zeta, \nu)
\quad \text{ for every } \om\in \Om, x\in A, \zeta \in \R^m_0 \text{ with }  |\zeta|\le 2k , \nu \in \mathbb S^{n-1}.
\end{equation*}
This implies that
\begin{equation}\label{G g tilde}
G(\om)(u,A)=\int_{S_u\cap A} g(\om,x,[u], \nu_u)\,d\mathcal H^{n-1}%(x) 
= \int_{S_u\cap A} \tilde g_k(\om,x,[u]\otimes \nu_u)\,d\mathcal H^{n-1}%(x)
\end{equation}
for every $u\in SBV%_{\mathrm{pc}}
(A,\R^m) \cap L^\infty(A,\R^m)$ with $\|u\|_{L^\infty(A,\R^m)}\le k$. 

Let $\alpha:=c_5/c_4\, (1+2 \|w\|_{L^\infty(A,\R^m)}) \,\mathcal H^{n-1}(S_w\cap A)$ as in Remark \ref{min-pc-bdd}. 
Given an increasing sequence $(A_j)$ of open sets, with $A_j\subset\subset A$  and $A_j \nearrow A$, we define
\begin{multline*}
\mathcal{X}^k_j:= \{u\in L^0(\R^n,\R^m) \colon u|_A \in SBV_{\mathrm{pc}}(A,\R^m) \cap L^\infty(A,\R^m),\
\|u\|_{L^\infty(A,\R^m)}\leq k,
\\ 
\mathcal{H}^{n-1}(S_u\cap A) \leq \alpha,\ u=w \; \textrm{in } A\setminus A_j \}.
\end{multline*}%
By \eqref{emme-kappa} we have
$$
\lim_{j\to +\infty} \inf_{u\in \mathcal{X}^k_j}G(\om)(u,A) = m^k_{G(\om)}(w,A).
$$
Therefore, to prove  \eqref{mk measurable}, and hence  the $\widehat\T$-measurability of $\om\mapsto m^{\mathrm{pc}}_{G(\om)}(w,A)$ it is enough to show that 
\begin{equation}\label{e:mj}
\om\mapsto \inf_{u\in \mathcal{X}^k_j}G(\om)(u,A) \textrm{ is } \;  \widehat\T\textrm{-measurable}.
\end{equation}
This will be obtained by using the Projection Theorem. To this end we consider $ \mathcal{X}^k_j$ as a topological space, with the topology induced by the weak$^\ast$-topology of $BV(A,\R^m)$, which is metrisable on $\mathcal{X}^k_j$. Indeed $BV(A,\R^m)$ is the dual of a separable space (see \cite[Remark 3.12]{AFP}), and $\mathcal{X}^k_j$ is bounded with respect to the $BV(A,\R^m)$-norm, since every $u\in \mathcal{X}^k_j$ satisfies 
$$
\|u\|_{BV(A,\R^m)} = \|u\|_{L^1(A,\R^m)}+|Du|(A) \leq k\mathcal{L}^n(A)+ 2k\alpha.
$$
Further, by virtue of Ambrosio's Compactness Theorem for $SBV(A,\R^m)$ (see \cite[Theorem 4.8]{AFP}), the topological space $\mathcal{X}^k_j$ is compact.
 
Let $\pi_\Om: \Om\times \mathcal{X}^k_j \to \Om$ be the canonical projection of $\Om\times  \mathcal{X}^k_j $ onto $\Om$. For every $t\in \R$ we have 
\begin{align*}
\left\{\om \in \Om: \inf_{u\in  \mathcal{X}^k_j}  G(\om)(u,A)<t\right\} = \pi_\Om\left(\{(\om,u)\in \Om\times  \mathcal{X}^k_j:  G(\om)(u,A)<t\}\right).
\end{align*}
By the Projection Theorem (see, e.g., \cite[Theorem III.13 and 33(a)]{Dellacherie}),
\eqref{e:mj} follows if we show that
\begin{equation}\label{e:joint_meas0}
(\om, u) \mapsto G(\om)(u,A) \; \textrm{ is } \;  \T\otimes {\B}(\mathcal{X}^k_j)\textrm{-measurable},
\end{equation}%
hence $\widehat\T\otimes {\B}(\mathcal{X}^k_j)$-measurable.

To prove this property we shall use \eqref{G g tilde}. By a Monotone Class argument (see the proof of Lemma~\ref{l:Measurability1})
we can assume, without loss of generality, 
that for every $\om \in \Omega$ and every $x\in \R^n$ the function $\xi \mapsto \tilde g_k(\omega, x, \xi)$ is continuous. 

In  \eqref{G g tilde} it is convenient to express $ [u]\otimes \nu_u$ and the restriction of $\mathcal H^{n-1}$ to $S_u$ by means of
the measure $\mu:=Du$.
By  \cite[Theorems~3.77 and~3.78]{AFP} for every  $B\in \B(A)$ we have
\begin{equation}\label{mu on Su}
\mu ( B)=\int_{ S_u\cap B} [u]\otimes \nu_u \,d\mathcal H^{n-1}\quad\text{and}\quad
|\mu| ( B)=\int_{ S_u\cap B}| [u]| \,d\mathcal H^{n-1},
\end{equation}
hence
\begin{equation}\label{Hn-1 and mu}
 \mathcal H^{n-1}(B)=\int_{ S_u\cap B}\frac{1}{| [u]|}\,d|\mu|.
\end{equation}

To write \eqref{G g tilde}  as a limit of measurable functions, for every $\mu \in \mathcal{M}_b(A,\R^{m\times n})$ and $\rho>0$ we consider the measure
$\mu^\rho\in \mathcal{M}_b(A,\R^{m\times n})$ defined by
\begin{equation*}
\mu^\rho (B):= \frac{\mu (B)}{\omega_{n-1}\rho^{n-1}} \quad \text{for every }B\in \B(A),
\end{equation*}%
where $\omega_{n-1}$ is the measure of the unit ball in $\R^{n-1}$. If $u\in SBV_{\mathrm{pc}}(A,\R^m)$ and $\mu=Du$, by the Besicovich Derivation Theorem and by the rectifiability of $S_u$ (see \cite[Theorems 2.22, 2.83, and 3.78]{AFP}) we deduce from \eqref{mu on Su} that,  when $\rho\to 0+$,
\begin{align}\label{limits on balls 1}
\mu^\rho ( B_\rho(x)\cap A)&\to( [u]\otimes \nu_u)(x)\quad\text{for }\mathcal H^{n-1}\text{-a.e.\ }x\in  S_u\cap A\,,
\\
|\mu^\rho |( B_\rho(x)\cap A)&\to |[u](x)|\quad\text{for }\mathcal H^{n-1}\text{-a.e.\ }x\in  S_u\cap A\,
\label{limits on balls 2}.
\end{align}

Since $\xi \mapsto \tilde g_k(\omega, x, \xi)$ is continuous and bounded uniformly with respect to $x$, by the Dominated Convergence Theorem it follows from \eqref{Hn-1 and mu}, \eqref{limits on balls 1}, and \eqref{limits on balls 2} that for every $u\in  \mathcal{X}^k_j$ we have 
\begin{equation}\label{lim eta rho}
G(\om)(u,A) = \lim_{\eta\to0+}  \lim_{\rho \to 0+} 
\int_A\frac{\tilde g_k\big(\om,x, \mu^\rho (A\cap B_\rho(x))\big)}{\max\{|\mu^\rho| (A\cap B_\rho(x)), \eta \}}\, d|\mu|(x),
\end{equation}%
with $\mu:=D u$. 
 Let $R:=2k\alpha$.  Since the map $u\mapsto Du$ from $BV(A,\R^m)$ into  $\mathcal{M}_b(A,\R^{m\times n})$ is continuous for the weak$^*$ topologies and the image of 
 $\mathcal{X}^k_j$ under this map is contained in $\mathcal M^R_A$, the claim in \eqref{e:joint_meas0} is an obvious consequence of \eqref{lim eta rho} and of the following property: for every $\eta>0$ and $\rho >0$ the function 
\begin{equation}\label{meas-1}
(\om, \mu) \mapsto \int_A\frac{\tilde g_k \big(\om,x, \mu^\rho (A\cap B_\rho(x))\big)}{\max\{ |\mu^\rho| (A\cap B_\rho(x)), \eta \}}\, d|\mu|(x)  \; \textrm{ is } \;  \T\otimes {\B}(\mathcal M^R_A)\textrm{-measurable}.
\end{equation}
To prove this property we observe that 
\begin{equation}\label{semicontinuity}
(x,\mu) \mapsto |\mu^\rho|(A\cap B_\rho(x))\quad \text{is (jointly) lower semicontinuous on }A\times \mathcal M^R_A\,.
\end{equation}
This is a consequence of the equality 
\begin{equation*}%\label{sup_lsc}
|\mu|(B_\rho(x)\cap A)= \sup\left\{ \int_{A}\varphi(y-x) d \mu (y): \varphi \in C_c^1(B_\rho(0),\R^{m{\times}n}),\ |\varphi|\leq 1 \right\}
\end{equation*}
and of the  (joint) continuity of
$
(x,\mu) \mapsto \int_{A}\varphi(y-x) d \mu (y)
$
on $A\times \mathcal M^R_A$.

We also observe that the $\R^{m{\times}n}$-valued function
\begin{equation}\label{measurability 77}
(x,\mu) \mapsto \mu^\rho(A\cap B_\rho(x))\quad \text{is }\B(A)\otimes {\B}(\mathcal M^R_A)\text{-measurable.}
\end{equation}
Indeed, given a nondecreasing sequence $(\varphi_j)$ of nonnegative functions in $C_c^1(B_\rho(0))$ converging to~$1$, we have
$$
\mu^\rho(A\cap B_\rho(x))=  \frac{1}{\om_{n-1}\rho^{n-1}}  \lim_{j\to+\infty}  \int_{A}\varphi_j(y-x) d \mu (y)\,,
$$
and each function $(x,\mu) \mapsto \int_{A}\varphi_j(y-x) d \mu (y)$
is (jointly) continuous on $A\times \mathcal M^R_A$.
Since $\tilde g_k$ is $\T\otimes\B(A)\otimes\B^{m{\times}n}$-measurable,
from \eqref{semicontinuity} and \eqref{measurability 77} we obtain that
\begin{equation*}
(\om, x, \mu) \mapsto \frac{\tilde g_k \big(\om,x, \mu^\rho (A\cap B_\rho(x))\big)}{ \max\{|\mu^\rho| (A\cap B_\rho(x)), \eta \}}  \; \textrm{ is } \;  
\T\otimes\B(A) \otimes{\B}(\mathcal M^R_A )\textrm{-measurable},
\end{equation*}
and \eqref{meas-1} follows from Corollary \ref{l:Measurability2}.
This concludes the proof.
\end{proof}

\section*{Acknowledgments}
 F. Cagnetti wishes to thank Panagiotis E. Souganidis for suggesting to consider the stochastic counterpart of \cite{CS}.  
The authors are grateful to Marco Cicalese who drew their attention to the argument in \cite[Proof of Theorem 5.5, Step 2]{Alciru} (see also the previous result by Braides and Piatnitski \cite[Proposition 2.10]{BP}) which is crucial in the proof of Theorem \ref{Ergodic-thm-surf2}.
F. Cagnetti was supported by the EPSRC under the Grant EP/P007287/1 ``Symmetry of Minimisers in Calculus of Variations''.
The research of G. Dal Maso was partially funded by the European Research Council under
Grant No. 290888 ``Quasistatic and Dynamic Evolution Problems in Plasticity
and Fracture''. G. Dal Maso is a member of the Gruppo Nazionale per l'Analisi Matematica, la Probabilit\`a e le loro Applicazioni (GNAMPA) of the Istituto Nazionale di Alta Matematica (INdAM).
L. Scardia acknowledges support by the EPSRC under the Grant EP/N035631/1 ``Dislocation patterns 
beyond optimality''.

%%%%%%%%%%%%%%%%%%%%%%%
%%%%%%%%%%%%%%%%%%%%%%%

\end{document}